\documentclass[11pt, reqno]{amsart}
\usepackage{latexsym, amsmath, amssymb, amsthm, mathtools}
\usepackage{cases, verbatim}
\usepackage[active]{srcltx}
\usepackage{esint}
\usepackage[colorlinks]{hyperref}
\usepackage{hypernat}
\usepackage{xcolor}
\usepackage{float}
\usepackage[normalem]{ulem}
\usepackage{cancel}
\usepackage{subcaption}

\usepackage[T1]{fontenc}
\usepackage{mathscinet}

\hypersetup{
    unicode=false,          
    pdftoolbar=true,        
    pdfmenubar=true,        
    pdffitwindow=false,     
    colorlinks=false,       
    linkcolor=red,          
    linkbordercolor=red,
    citecolor=green,        
    citebordercolor=green,
    filecolor=magenta,      
    urlcolor=cyan,           
   urlbordercolor={1 1 1}  
}

\usepackage[margin=1.5in]{geometry}

\newtheorem{thm}{Theorem}
\newtheorem{lem}[thm]{Lemma}
\newtheorem{prop}[thm]{Proposition}

\theoremstyle{remark}
\newtheorem{rmk}[thm]{Remark}
\newtheorem{example}[thm]{Example}

\theoremstyle{definition}
\newtheorem{defi}[thm]{Definition}

\numberwithin{thm}{section} 
\numberwithin{equation}{section}

\makeatletter

\newcommand{\Rmnum}[1]{\expandafter\@slowromancap\romannumeral #1@}
\makeatother

\def\A{{\mathcal A}}
\def\R{{\mathbb R}}
\def\X{{\mathbf X}}
\def\Y{{\mathbf Y}}

\newcommand{\vep}{\varepsilon}
\newcommand{\ol}{\overline}

\newcommand{\bpm}{\begin{pmatrix}}
\newcommand{\epm}{\end{pmatrix}}

\newcommand{\la}{\langle}
\newcommand{\ra}{\rangle}

\newcommand{\beq}{\begin{equation}}
\newcommand{\eeq}{\end{equation}}

\def\lip{{\rm Lip\,}}
\def\lipl{{\rm Lip_{loc}\,}}

plus 6pt minus 12pt
\title[Monge solutions of time-dependent Hamilton-Jacobi equations]{Monge solutions of time-dependent Hamilton-Jacobi equations in metric spaces}

\author[Q.~Liu]{Qing Liu}
\address[Qing Liu]{Geometric Partial Differential Equations Unit, Okinawa Institute of Science and Technology Graduate University, Okinawa 904-0495, Japan}
\email{qing.liu@oist.jp}

\author[M.~B.~P.~Wiranata]{Made Benny Prasetya Wiranata}
\address[Made Benny Prasetya Wiranata]{Geometric Partial Differential Equations Unit, Okinawa Institute of Science and Technology Graduate University, Okinawa 904-0495, Japan, and \newline
Faculty of Mathematics and Natural Sciences, Universitas Gadjah Mada, Indonesia}

\email{made.wiranata@oist.jp}

\date{\today}

\begin{document}

\begin{abstract}
 As a classical notion equivalent to viscosity solutions, Monge solutions are well understood for stationary Hamilton-Jacobi equations in Euclidean spaces and have been recently studied in general metric spaces. In this paper, we introduce a notion of Monge solutions for time-dependent Hamilton-Jacobi equations in metric spaces. 
The key idea is to reformulate the equation as a stationary problem under the assumption of Lipschitz regularity for the initial data. We establish the uniqueness and existence of bounded Lipschitz Monge solutions to the initial value problem and discuss their equivalence with existing notions of metric viscosity solutions. 

\end{abstract}


\maketitle



\section{Introduction}
\subsection{Background and motivation}
In recent years, there have been remarkable developments in the study of first-order Hamilton-Jacobi equations in general metric spaces, driven by various important applications in optimal transport \cite{AGS13, Vbook}, mean field games \cite{CaNotes}, traffic flow and networks \cite{IMZ,  ACCT,  IMo1, IMo2} etc. In the Euclidean space, it is well known that the viscosity solution theory provides a general framework for solving Hamilton-Jacobi equations, including the typical form 
\[
\partial_t u+H(x, \nabla u)=0;
\]
see, for example, \cite{CIL, BCBook, KoBook, TrBook} for comprehensive introductions. We also refer to \cite{MS, Dr} for extensions of the viscosity solution approach to Hamilton-Jacobi equations on sub-Riemannian manifolds. By contrast, the study of fully nonlinear first-order PDEs in a general metric space $(\X, d)$ faces significant challenges, primarily due to the lack of a linear structure in $\X$ and the difficulty of defining $\nabla u$, the gradient of the unknown function. Note that while $C^1$ functions are used as test functions to define viscosity solutions in the Euclidean space, such a test class is no longer available in general metric spaces. 

Notable progress has been made in the case when the Hamiltonian $H$ depends on $\nabla u$ in terms of its norm $|\nabla u|$. Two types of metric viscosity solutions are introduced in \cite{GHN, Na1} and in \cite{AF, GaS2, GaS} to handle such equations. Here, let us briefly outline these approaches in the special case of the eikonal equation 
\begin{equation}\label{eikonal}
|\nabla u|=f(x) \quad \text{in $\Omega$,}
\end{equation}
where $\Omega\subset \X$ is a domain and $f$ is continuous in $\Omega$. 
The notion of viscosity solutions proposed by \cite{GHN}, which we call curve-based viscosity solutions in this work, reduces \eqref{eikonal} to a one-dimensional problem by using optimal control interpretations along absolutely continuous curves in $\X$. Roughly speaking, using the composition $w_\gamma=u\circ \gamma$ with arc-length parametrized curves $\gamma$ in $\X$, one can interpret $|\nabla u|(x)$ as the maximum value of the derivative $w_\gamma'(0)$ if $\gamma(0)=x$. This perspective allows us to define viscosity solutions of \eqref{eikonal} in general metric spaces as in \cite{GHN} by adopting tests for $w_\gamma$ in one dimension rather than $u$ itself. 

Another notion of solutions, referred to as slope-based viscosity solutions in this paper, is introduced in \cite{AF, GaS2, GaS}, generalizing the Euclidean viscosity solution theory in a more direct manner. The key idea is to employ $\psi(x)=\kappa d(x, x_0)^2$, for any fixed $x_0\in \X$,  as the base of test functions for subsolutions (resp., supersolutions) with $\kappa\geq 0$ (resp., $\kappa \leq 0$); see \cite{GaS2, GaS} and Definition \ref{def s-sol} below for a more precise definition of test functions. Interpreting $|\nabla \psi|$ as the (local) slope of $\psi$, defined by 
\[
|\nabla \psi|(x)= \limsup_{y\to x} \frac{|\psi(x)-\psi(y)|}{d(x, y)}
\]
for $x\in \Omega$, we see that $|\nabla \psi|=2|\kappa|d(\cdot, x_0)$ is continuous in $\X$. This substitute of $C^1$ test class turns out to be sufficient to establish uniqueness and existence of viscosity solutions to a general class of Hamilton-Jacobi equations including \eqref{eikonal} in a complete length space. Recall that $(\X, d)$ is called a length space if for all $x, y\in \X$, $d(x, y)$ coincides with the infimum of $\ell(\gamma)$ of all rectifiable curves $\gamma$ in $\X$ joining $x, y$, where $\ell(\gamma)$ denotes the length of $\gamma$. Consult \cite{NN, Mak, MaNa} for stability results and \cite{LNa} for convexity-preserving properties about slope-based viscosity solutions. 

While the aforementioned notions of viscosity solutions appear different, they are proved in \cite{LShZ} to be equivalent for a class of stationary problems such as \eqref{eikonal} in a complete length space. The equivalence is established through a third notion of solutions known as Monge solutions, 
which offers a convenient definition that does not rely on test functions. Defining Monge solutions of the eikonal equation \eqref{eikonal} is particularly simple. Let $\Omega$ be a domain in a complete length space $(\X, d)$. Denote by $\lipl(\Omega)$ the set of all locally Lipschitz functions. A function $u\in \lipl(\Omega)$ is said to be a Monge solution if $|\nabla^- u|(x)=f(x)$ at any $x\in \Omega$, where $|\nabla^- u|(x)$ denotes the subslope of $u$ at $x$, defined as 
\[
|\nabla^- u|(x)= \limsup_{y\to x} \frac{\max\{u(x)-u(y), 0\}}{d(x, y)}. 
\]
Utilizing the subslope instead of the full slope is essential, especially when considering the uniqueness problem and the consistency with viscosity solutions. See \cite{DaSa, DaLeSa} for recent applications of subslope in metric determination. 

The notion of Monge solutions is first proposed by \cite{NeSu} as an alternative way to understand viscosity solutions of stationary Hamilton-Jacobi equations in the Euclidean space. Since its definition does not involve the use of test functions, this approach can be extended relatively easily to discontinuous Hamilton-Jacobi equations. We refer to \cite{NeSu, CaSi1, BrDa} for related results in the Euclidean space and to \cite{EGV} for recent progress on sub-Riemannian manifolds. In the context of general metric spaces, discontinuous eikonal equations have been studied recently in \cite{LShZ2} based on the techniques for Monge solutions. See also \cite{LMit1} for an application of the Monge-solution approach to the eigenvalue problem for infinity Laplacian in metric spaces.

The purpose of this paper is to develop the theory of Monge solutions in a complete length space for the time-dependent Hamilton-Jacobi equation of the form
\begin{equation}\label{eq general}
\partial_t u+H(x, t, |\nabla u|)=0 \quad \text{in $\X\times (0, T)$}
\end{equation}
for $T>0$ given, along with a bounded continuous initial value
\begin{equation}\label{initial}
u(x, 0)=u_0(x),\quad \text{$x\in \X$.}
\end{equation} 
To the best of our knowledge, the Monge approach to such time-dependent problems has not been available even in Euclidean spaces. In this work, under suitable assumptions on $H$ including the monotonicity, convexity, and coercivity of $p\mapsto H(x, t, p)$, we introduce a generalized notion of Monge solutions of \eqref{eq general} and establish the uniqueness and existence of Monge solutions to the associated initial value problem. 

Moreover, as in the stationary case, we examine the relation between our Monge solutions and the other notions of solutions to the time-dependent problem. When $H$ is independent of $t$, the existence and uniqueness of slope-based solutions of \eqref{eq general} are established in \cite{AF, GaS2, GaS}, while \cite{Na1} provides well-posedness results for a notion of curve-based solutions that generalizes the stationary counterpart in \cite{GHN}; see  Section~\ref{sec:prelim} for a brief review of these results. 
The equivalence between them, however, seems to remain unaddressed, although it is expected to hold, especially from the control-theoretic perspective. In this work, we give an affirmative answer to this important question, rigorously proving that the curve-based, slope-based, and Monge solutions to the initial value problem for \eqref{eq general} are all equivalent provided that the time-independent $H$ and the initial value satisfy appropriate assumptions that will be specified later. 

It is worth emphasizing that our current work only focuses on Hamilton-Jacobi equations in complete length spaces with Hamiltonians depending on $|\nabla u|$. For more general Hamiltonians that depend on $\nabla u$ instead, additional structures or assumptions on the metric space are often needed. There is vast literature on such general Hamilton-Jacobi equations in settings like networks or Wasserstein spaces. We refer the reader to \cite{GaNT, CCM1, HK, IMZ, IMo1, IMo2, GaMS, JZ} and references therein for developments in these directions.

\subsection{Time-dependent Monge solutions}\label{sec:monge-notion}
Let us explain the heuristics behind our definition of time-dependent Monge solutions. We begin with discussing the simple eikonal-type Hamiltonians, that is, $H(x, t, p)=|p|-f(x, t)$ for a given bounded function $f\in C(\X\times (0, T))$. Here and in the sequel, we fix $T>0$ and assume that $(\X, d)$ is a length space. In this case, the equation reads
\begin{equation}\label{eq special}
    \partial_t u+|\nabla u|=f(x, t) \quad \text{in $\X\times(0,T)$}.
\end{equation}
Our strategy is straightforward: we treat the time and space variables jointly, reformulating \eqref{eq special} as a stationary equation in the product metric space $\Y_T:=\X\times (0, T)$. Heuristically speaking, to adapt the method in \cite{LShZ} for stationary eikonal equation to this new formulation, it is desirable to have a nonnegative term of $\partial_t u$ so that the left hand side of \eqref{eq special} can be rewritten as $|\partial_t u|+|\nabla u|$, which represents a certain gradient norm of $u$ in space-time. To this end, we assume that $u\in S_k(\Y_T)$ for some $k\geq -\inf_{\Y_T}f$, where 
\begin{equation}\label{sk}
\begin{aligned}
S_k(\Y_T)=\big\{u\in \lipl(\Y_T): u(x, t) -u(x, s) &\geq -k(t-s),\\
& \text{for all } x\in \X,  0<s<t<T\big\},
\end{aligned}
\end{equation}
and take
\begin{equation}\label{unknown-change}
v(x, t)=u(x, t)+kt, \quad x\in \X,\ t\in [0, T). 
\end{equation}
This change of unknown formally yields $\partial_t v\geq 0$ and converts \eqref{eq special} to 
\begin{equation}\label{eq special2}
|\partial_t v|+|\nabla v|=f(x, t)+k \quad \text{in $\Y_T$.}
\end{equation}
In light of our treatment for the stationary eikonal equation, the left hand side of \eqref{eq special2} is naturally understood as the subslope of $v$ in $\Y_T$ with metric $\bar{d}$ defined as 
\begin{equation}\label{metric spacetime}
\bar{d}((x_1, t_1), (x_2, t_2)):=  \max\{d(x_1, x_2), |t_1-t_2|\} 
\end{equation}
for $(x_1, t_1), (x_2, t_2)\in \Y_T=\X\times (0, T)$. 
We define Monge solutions of \eqref{eq special} as $u\in S_k(\Y_T)$ such that $v$ given by \eqref{unknown-change} satisfies 
\begin{equation}\label{monge-eikonal-intro}
|D^-v|(x, t)=f(x, t)+k, \quad \text{for all $(x, t)\in \Y_T$, }
\end{equation}
where we take the subslope $|D^-v|$ (in space-time) by 
\begin{equation}\label{def D-0}
    |D^- v|(x, t)= \limsup_{\delta \to 0+}\sup\left\{\frac{\max\{v(x, t)-v(y, s), 0\}}{\delta}: (y,s)\in \Y_T, \ \bar{d}((x, t), (y, s))\leq \delta\right\}. 
\end{equation}
An equivalent expression of $|D^- v|$ under the monotonicity of $v$ in time is given in \eqref{def D-}. Note that no test functions are utilized in this definition of solutions. 
The choice of metric $\bar{d}$ for the product space $\X\times (0, T)$ is consistent with the $L^1$-$L^\infty$ duality in the Euclidean space. Suppose that $v\in C^1(\R^n\times (0, T))$ is nondecreasing in time and $(x, t)\in \R^n\times (0, T)$. Then by Taylor expansion, we have 
\[
v(y, s)-v(x, t)=\partial_t v(x, t)(s-t)+\la\nabla v(x, t), {y-x}\ra+ o(|t-s|+|x-y|)
\]
for all $(y, s)$ near $(x, t)$. Since $\partial_t v(x, t)\geq 0$, it then follows from the H\"older inequality that 
\[
\begin{aligned}
\max\{v(x, t)-v(y, s), 0\}\leq \left(\partial_t v(x, t)+|\nabla v(x, t)|\right) & \max\left\{|t-s|, {|x-y|}\right\}\\
&\quad  + o(|t-s|+|x-y|),
\end{aligned}
\]
and moreover,  
\[
\partial_t v(x, t)+|\nabla v(x, t)|=\limsup_{(y, s)\to (x, t)} \frac{\max\{v(x, t)-v(y, s), 0\}}{\max\left\{|t-s|, {|x-y|}\right\}}. 
\]
The relation above gives an intuitive interpretation for the subslope $|D^-v|(x, t)$ appearing in \eqref{def D-0}. As shown later in Theorem~\ref{thm:into1}, the requirement $u\in S_k(\Y_T)$ imposed earlier can be fulfilled if $u_0$ is assumed to be bounded and Lipschitz in a complete length space $\X$. This justifies the entire structure of the well-posedness argument for Monge solutions. 

It is worth mentioning that the condition \eqref{monge-eikonal-intro} is a property on $u$ itself and does not depend on any particular choice of $k\geq 0$ among those satisfying $u\in S_k(\Y_T)$ and $k+\inf_{\Y_T}\geq 0$. See Section~\ref{sec:eikonal-monge-def} for more discussions about the definition of Monge solutions of \eqref{eq special}.

For time-dependent Hamilton-Jacobi equations other than the eikonal type, our definition of Monge solutions becomes more intricate due to the loss of degree-$1$ homogeneity of $p\mapsto H(x, t, p)$ even when $H$ is independent of $x, t$. A typical example in this case is 
\begin{equation}\label{2-homo-eq0}
\partial_t u+\frac{1}{2}|\nabla u|^2=f(x) \quad \text{in $\Y_T=\X\times (0, T)$, }
\end{equation}
for which the Hamiltonian $H$ is given by $H(x, p)=p^2/2-f(x)$. Assuming that $u\in S_k(\Y_T)$, we still obtain 
\begin{equation}\label{2-homo-eq}
|\partial_t v|+\frac{1}{2}|\nabla v|^2-f(x)=k \quad \text{in $\Y_T$, }
\end{equation}
for $v$ given by \eqref{unknown-change}. Unlike \eqref{eq special2} in the eikonal case, the left-hand side of \eqref{2-homo-eq} can no longer be viewed as a slope, since its homogeneity in $|\partial_t v|$ and in $|\nabla v|$ does not match. However, such a time-dependent problem \eqref{2-homo-eq} can still be treated as a stationary problem in the space-time product space with the left-hand side as a new Hamiltonian $H'(x, \tau, p)=\tau+p^2/2-f(x)$, where the pair $(\tau, p)\in [0, \infty)^2$ serves as the momentum variable. Using the standard Lagrangian formulation, we see that 
\[
H'(x, \tau, p)=\sup_{q\geq 0} \left\{\tau+pq-\frac{q^2}{2}-f(x)\right\},\quad x\in \X, \tau, p\geq 0. 
\]
A formal Taylor expansion, as in the case of the eikonal equation, then suggests defining Monge solutions of \eqref{2-homo-eq0} by requiring 
\begin{equation}\label{monge-intro}
|D_L^- v|(x, t)=k \quad \text{for all $(x, t)\in \Y_T$}
\end{equation}
with $|D_L^- v|$ given as 
\[
|D_L^- v|(x, t)= \limsup_{(y, s)\to (x, t)} \left\{\frac{\max\{v(x, t)-v(y, s), 0\}}{|t-s|} -\frac{d(x, y)^2}{2|t-s|^2}-f(x)\right\}, \quad (x, t)\in \Y_T. 
\]
The subslope-like quantity $|D_L^- v|$ is used to maintain consistency with the corresponding viscosity solutions.

The same strategy can be applied to more general Hamilton-Jacobi equations as in  \eqref{eq general}. Assuming that $p\mapsto H(x, t, p)$ is increasing, convex, and coercive in $[0, \infty)$ with superlinear growth, we define Monge solutions $u$ in the class $S_k(\Y_T)$ for some $k\geq 0$ with the same condition \eqref{monge-intro} on $v$ given by \eqref{unknown-change}. In this general case, $|D_L^- v|$ is defined by
\[
|D_L^- v|(x, t)= \limsup_{(y, s)\to (x, t)} \left\{\frac{\max\{v(x, t)-v(y, s), 0\}}{|t-s|} -L\left(x, t, \frac{d(x, y)}{|t-s|}\right)\right\}, \quad (x, t)\in \Y_T,  
\]
where $L: \X\times (0, T)\times [0, \infty)\to \R$ denotes the Lagrangian associated to $H$, that is, 
\begin{equation}\label{legendre0}
 L(x, t, q)= \sup_{p\geq 0} \{p q- H(x, t, p)\}\quad \text{for all $x\in \X$, $t\in (0, T)$, $q\geq 0$.}
\end{equation}
Under the monotonicity of $t\mapsto v(x, t)$, one can alternatively define $|D_L^- v|$ to be 
\begin{equation}\label{def DL-into}
|D_L^- v|(x, t)= \limsup_{y\to x,\ s\to t-} \left\{\frac{v(x, t)-v(y, s)}{t-s} -L\left(x, t, \frac{d(x, y)}{t-s}\right)\right\}, \quad (x, t)\in \Y_T. 
\end{equation}
Our definition of Monge solutions is also compatible with the classical Hopf-Lax formula (adapted to a metric space $(\X, d)$)
\[
u(x, t)=\inf\left\{u_0(y)+t L\left(\frac{d(x, y)}{t}\right)\right\}, \quad \text{$(x, t)\in \X\times (0, T)$,}
\]
in the special case where the Hamiltonian $H$ and the corresponding Lagrangian $L$ are independent of $x$ and $t$. We refer to \cite[Theorem~7.7]{GaS2} for the Hopf-Lax formula of slope-based viscosity solutions in general geodesic spaces. Since the formula implies 
\[
u(x, t)=\inf_{y\in \X,\ 0\leq s<t}\left\{u(y, s)+(t-s) L\left(\frac{d(x, y)}{t-s}\right)\right\}, 
\]
the function $v$ defined by \eqref{unknown-change} satisfies 
\[
\sup\left\{\frac{v(x, t)-v(y, s)}{t-s}-L\left(\frac{d(x, y)}{t-s}\right): y\in \X,\ 0\leq s<t\right\}=k.
\]
We immediately obtain \eqref{monge-intro} by letting $y\to x, s\to t-$. Hence, \eqref{monge-intro} can be somewhat interpreted as a pointwise infinitesimal formulation of the Hopf-Lax formula, applicable to a broad class of Hamiltonians that also depend on $x$ and $t$.

While Hamilton-Jacobi equations have been extensively studied, especially within the framework of viscosity solutions, to the best of our knowledge, this notion of Monge solutions, treating space and time as a whole, seems to be new even in the Euclidean setting. In this work, we establish well-posedness results for Monge solutions to the initial value problem and discuss the equivalence with other types of solutions under further assumptions on $H$. More details will be given later. 

\subsection{Main results}

To maintain consistency with our definition of Monge solutions, we divide our discussion of the well-posedness of the initial value problem into two cases as well. We study the eikonal-type \eqref{eq special} first and then discuss \eqref{eq general} with $H$ convex, coercive, and growing superlinearly in $p$. In both cases, we assume that $u_0\in \lip(\X)$, where $\lip(\X)$ denotes the set of all Lipschitz functions in $\X$.  

Our main result in this case of \eqref{eq special} is as below. 
\begin{thm}[Well-posedness for Monge solutions of eikonal-type equations]\label{thm:into1}
     Let $(\X,d)$ be a complete length space. Let $T>0$. Assume that $u_0\in \lip(\X)$ is bounded. Assume in addition that $f\in C(\X\times (0, T))$ is bounded and $f(x, t)$ is Lipschitz with respect to either $x$ or $t$. Then, there exists a unique bounded Monge solution $u\in \lip(\X\times [0, T))$ of \eqref{eq special} satisfying \eqref{initial} in the sense of 
     \begin{equation}\label{initial-conti0}
     \sup_{x\in \X} |u(x, t)-u_0(x)|\to 0 \quad \text{as $t\to 0+$.}
     \end{equation}
 \end{thm}

Since \eqref{eq special} is now understood as a stationary eikonal equation \eqref{eq special2} as explained in Section~\ref{sec:monge-notion}, we essentially follow the approach in \cite{LShZ} to prove this theorem, particularly the comparison principle. The existence of Monge solutions is proved by adapting the classical control-theoretic interpretation for eikonal equations. In fact, we show that $u$ given by 
\begin{equation}\label{eq value_f0}
    u(x,t) = \inf_{\gamma\in \Gamma_{x}}\left\{ 
u_0(\gamma(t))+\int_0^t f(\gamma(\sigma),t-\sigma)\;d\sigma \right\},\quad (x,t)\in \X\times[0,T)
\end{equation} 
is a Monge solution of \eqref{eq special} satisfying \eqref{initial-conti0}, where, for any $x\in \X$,  
\begin{equation}\label{eikonal-curve}
\begin{aligned}
    \Gamma_{x}:=& \bigg\{ \gamma:[0,\infty)\to \X: \ \text{$\gamma$ is Lipschitz, $\gamma(0)=x$, and } |\gamma'|\leq 1\ \text{a.e. in } (0,\infty)\bigg\}. 
\end{aligned}
\end{equation}
In contrast to the usual arguments for verifying the viscosity solution property of a value function, an additional step here is to show that $u\in S_k(\Y_T)$ for a certain  $k\geq 0$. This is achieved by proving the Lipschitz continuity of $u$ in $\Y_T$, adopting the regularity assumptions on $u_0$ and $f$ stated in Theorem~\ref{thm:into1}.  

Analogous well-posedness results can be obtained for \eqref{eq general}, where $H\in C(\X\times (0, T)\times [0, \infty))$ is assumed to satisfy a set of conditions as listed below. We first assume that 
\begin{enumerate}
\item[(H1)] $p\mapsto H(x, t, p)$ is convex and nondecreasing in $[0, \infty)$ for any $x\in \X$ and $t\in (0, T)$. It is also coercive in the sense that
    \begin{equation}\label{H-cocercive}
\inf\left\{ \frac{H(x, t, p)}{p}: (x, t)\in \X\times (0, T),\ p\geq R\right\}\to \infty\quad \text{as $R\to \infty$.}
\end{equation}
\end{enumerate}
Let $L: \X\times (0, T)\times [0, \infty)\to \R$ be defined as in \eqref{legendre0}.
We see that $q\mapsto L(x, t, q)$ is convex and nondecreasing in $[0, \infty)$. We also obtain
\begin{equation}\label{legendre}
    H(x, t, p)=\sup_{q\geq 0}\{p q- L(x, t, q)\}, \quad\text{for $x\in \X$, $t\in (0, T)$, $p\geq 0$,}
\end{equation}
whose proof is similar to that of \cite[Proposition 2.1]{Na1}. More assumptions on $L$ are as follows. 
\begin{enumerate}
\item[(H2)] There exists a convex function $m\in C(\R)$ such that
    \begin{equation}\label{L-lower}
    \inf_{(x, t)\in \X\times (0, T)}L(x, t, q)\geq m(q) \quad\text{ for all $q\geq 0$. }
    \end{equation}
    and 
    \begin{equation}\label{L-cocercive}
  \frac{m(q)}{q}\to \infty \quad \text{as $q\to \infty$}.
\end{equation}
\item[(H3)] $\sup\limits_{(x, t)\in\X\times (0,T)}L(x, t, 0)< +\infty$.

\smallskip 

\item[(H4)] $L$ is locally uniformly continuous in $\X\times (0, T)\times [0, \infty)$, and there exists a modulus of continuity $\omega_L$ such that 
    \begin{equation}\label{L-lip1}
    \begin{aligned}
    |L(x,t, q)-L(y,s, q)|\leq &\  \omega_L(d(x,y)+|t-s|)(1+|m(q)|) \\
    & \text{for all $x, y\in \X$, $t, s\in (0, T)$, $q\geq 0$,}
    \end{aligned}
    \end{equation}
where $m\in C(\R)$ is given in (H2). 

\smallskip 

\item[(H5)] $L$ is Lipschitz continuous with respect to the time variable; namely, there exists $C_T>0$ such that 
    \begin{equation}\label{L-lip2}
    |L(x, t, q)-L(x, s, q)|\leq C_T |t-s| \quad \text{for all $x\in \X$, $t, s\in (0, T)$, $q\geq 0$.}
    \end{equation}
\end{enumerate}   
Let us state our main result in this case. 
\begin{thm}[Well-posedness for Monge solutions of superlinear Hamilton-Jacobi equations]\label{thm:intro2}
     Let $(\X,d)$ be a complete length space and $T>0$. 
     Assume that $H\in C(\X\times (0, T)\times [0, \infty))$ satisfies (H1) and $L\in C(\X\times (0, T)\times [0, \infty))$ given by \eqref{legendre0} satisfies (H2)--(H5). Let $u_0\in\lip(\X)$ be bounded.   Then, there exists a unique bounded Monge solution $u\in \lip(\X\times [0, T))$ of \eqref{eq general} satisfying \eqref{initial-conti0}.  
\end{thm}
As an advantage of utilizing the notion of Monge solutions without requiring any test functions, our proof of comparison principle for both eikonal-type and superlinear Hamilton-Jacobi equations does not rely on the doubling variable technique, which is typically required in the standard viscosity argument. For this reason, we do not need the continuity of $H$ with respect to $x$ and $t$ (related to (H4)(H5) above) to prove the comparison principle. Although this work focuses solely on the case of continuous Hamiltonians, the same approach can be extended to handle discontinuous Hamilton-Jacobi equations, as shown in \cite{LShZ2} for stationary discontinuous eikonal equations.  Discontinuous evolution Hamilton-Jacobi equations in the Euclidean space have been studied with different methods; see, for example, \cite{Cam, CaSi, Dav, ChHu, GHa}. Comparing our Monge approach with these methods merits further study in future work. 

The existence of Monge solutions in this case is again based on a generalization of the classical control-theoretic (variational) interpretation of \eqref{eq general}. Our convenience in showing the uniqueness of Monge solutions comes at the cost of an additional regularity assumption $u_0\in \lip(\X)$, which enables us to show $u\in \lip(\Y_T)$ and therefore $u\in S_k(\Y_T)$ for some $k\geq 0$. 

In addition to the uniqueness and existence results for Monge solutions, we also discuss the connection with other existing approaches. We focus on the special case when $H$ is independent of $t$, i.e., $H=H(x, p)$, for which the equation \eqref{eq general} reduces to 
\begin{equation}\label{eq time-ind}
  \partial_t u+H(x, |\nabla u|)=0 \quad \text{in $\X\times (0, T)$.}
\end{equation}
We can compare the three notions in this case, since the well-posedness results for curve-based and slope-based viscosity solutions are explicitly available in \cite{Na1, GaS} in this setting. Our equivalence result applies to the initial value problem for 
the following typical Hamilton-Jacobi equation:
\begin{equation}\label{eq power}
    \partial_t u+a(x) |\nabla u|^\alpha=f(x) \quad \text{in $\X\times (0, T)$,}
\end{equation}
where $\alpha>1$ is given, and $a, f: \X\to \R$ are assumed to be bounded uniformly continuous with $\inf_\X a>0$. 
\begin{thm}[Equivalence of curve-based, slope-based, and Monge solutions]\label{thm:equiv-power}
 Let $(\X,d)$ be a complete length space and $T>0$. Let $\alpha>1$ and $a, f: \X\to \R$ be bounded uniformly continuous functions such that $\inf_\X a>0$. Assume that $u_0\in \lip(\X)$ is bounded and $u\in C(\X\times (0, T))$ satisfies \eqref{initial-conti0}. Then the following statements are equivalent. 
\begin{enumerate}
     \item[(i)] $u$ is a curve-based viscosity solution of \eqref{eq power};
     \item[(ii)] $u$ is a slope-based viscosity solution of \eqref{eq power};
     \item[(iii)] $u$ is a Monge solution of \eqref{eq power}.
 \end{enumerate}
\end{thm}
Since the Monge solution is shown to be Lipschitz continuous in $\X\times [0, T)$ in Theorem~\ref{thm:intro2}, as a consequence of Theorem~\ref{thm:equiv-power}, we see that both the curve-based and slope-based viscosity solutions of the initial value problem for \eqref{eq power} are Lipschitz continuous as well. 

The equivalence result in Theorem~\ref{thm:equiv-power} is an application of our discussions in Section~\ref{sec:equivalence} to the special case \eqref{eq power}. Further equivalence results for \eqref{eq time-ind} with more general superlinear Hamiltonians are given in Theorem~\ref{thm:equiv-cm} and Theorem~\ref{thm:equivalence1}. Using similar arguments, these three notions of solutions are also shown to be equivalent for the eikonal-type equation 
\begin{equation}\label{eq special0}
    \partial_t u+|\nabla u|=f(x) \quad \text{in $\X\times(0,T)$,}
\end{equation}
where $f$ is assumed to be bounded and uniformly continuous in $\X$. See Theorem~\ref{thm:equiv-eikonal1} and Theorem~\ref{thm:equiv-eikonal2} for details. Our current equivalence results rely on the comparison principle for the initial value problem, assuming the Lipschitz regularity of the initial data. It is of our interest in future work to study the local equivalence between these notions in $\X\times (0, T)$ without using the initial value. 

We remark that our notion of Monge solutions, along with the associated well-posedness and equivalence results, can be extended to a broader class of convex Hamilton-Jacobi equations in \eqref{eq time-ind}, including Hamiltonians that fail to satisfy the strong coercivity condition \eqref{H-cocercive} and exhibit a more general linear growth in $p$ than  \eqref{eq special}. However, for simplicity of our presentation in this paper, we will not pursue a detailed study on this case, since it involves more technical complications in dealing with a more general, possibly infinity-valued Lagrangian $L$. 

This paper is organized as follows. Section~\ref{sec:prelim} provides an overview of preliminaries on Hamilton-Jacobi equations in metric spaces, including the notions and some basic properties of the curve-based and slope-based viscosity solutions. In Section~\ref{sec:eikonal}, we introduce the definition of Monge solutions to \eqref{eq special} and prove the well-posedness result stated in Theorem~\ref{thm:into1} for the initial value problem. Our study of Monge solutions to \eqref{eq general} under (H1)--(H5), including the proof of Theorem~\ref{thm:intro2}, is presented in Section~\ref{sec:superlinear}. Finally, Section~\ref{sec:equivalence} explores the relations between the curve-based, slope-based solutions and Monge solutions. 

\subsection*{Acknowledgments}

The work of QL was supported by JSPS Grant-in-Aid for Scientific Research (No.~22K03396).



\section{Preliminaries}\label{sec:prelim}

In this section, we briefly review two notions of metric viscosity solutions of \eqref{eq time-ind} and the related well-posedness results for the eikonal-type equation \eqref{eq special0} and the superlinear equation \eqref{eq general} (with $H$ independent of $t$ satisfying (H1)--(H4)), which are of our particular interest in this work. The first notion is introduced in \cite{Na1} and based on an earlier work \cite{GHN} on the stationary eikonal equation. The primary idea is to apply the optimal control interpretation and consider the composition of functions and curves in $\X$, reducing the problem to one space dimension. The second type, proposed by \cite{AF, GaS2, GaS}, employs appropriate test functions to substitute the $C^1$ class in the Euclidean spaces, thus extending the standard viscosity solution theory to metric spaces.

For our convenience in later reference, we will call the solutions studied in \cite{Na1} curve-based (viscosity) solutions and those in \cite{AF, GaS2, GaS} slope-based (viscosity) solutions.
We will discuss the relation between these notions and the Monge solution in Section~\ref{sec:equivalence}. 

\subsection{Curve-based viscosity solutions}

Let us begin with preliminaries about the curve-based solutions of \eqref{eq time-ind} introduced in \cite{Na1}. We assume that $H\in C(\X\times [0, \infty))$ satisfies (H1). Let $L: \X\times [0, \infty)\to \R$ be the Lagrangian defined by \eqref{legendre0} (without $t$-dependence). Note that without further assumptions on $H$ or $L$, such a function $L$ may take infinity values. One typical example is the eikonal-type equation \eqref{eq special0}, which gives, for $x\in \X$, $p, q\geq 0$, 
\begin{equation}\label{eikona-hl}
H(x, p)=p-f(x), \qquad L(x, q)=\begin{cases}
    f(x) \quad & \text{for}\ 0\leq q\leq 1, \\
    \infty & \text{for } q>1. 
    \end{cases}
\end{equation}
The following curve class is used in \cite{Na1}. 
\begin{defi}
Let $AC(I,\X)$ denote the set of all absolutely continuous curves in $\X$ defined on an interval $I$ of $\R$. Let $\mathcal{A}(\X)$ be the set of all admissible curves $\gamma\in AC([0,\infty),\X)$ such that the curve speed $|\gamma'|$ is piecewise constant and $L(\gamma(\sigma), t-\sigma, |\gamma'|(\sigma))$ is piecewise continuous, that is, $\sigma \mapsto |\gamma'|(\sigma)$ equals a constant $v_I$ a.e. and $\sigma\mapsto L(\gamma(\sigma), t-\sigma, v_I)$ is continuous on each $I=[0,r_1],[r_1,r_2],\ldots,[r_n,\infty)$ with finitely many $r_1,\ldots,r_n$. For $x\in\X$, we denote 
\begin{equation}\label{admissible_set_gen}
\mathcal{A}_x(\X)=\{\gamma\in\mathcal{A}(\X)\;|\;\gamma(0)=x\}.    
\end{equation}
\end{defi}

We omit the original definition of curve-based solutions since it involves more notations and technical details. Instead, we take from \cite{Na1} a convenient characterization of curve-based subsolutions and supersolutions established via the Lagrangian $L$. Below, $u$ is said to be an arcwise continuous function in $\X\times(0,T)$ if for every $\gamma\in AC(\R,\X)$, the function $w(s,t)=u(\gamma(s),t)$ is continuous in $\R\times(0,T)$. 

\begin{prop}[\cite{Na1}, Proposition 3.1]\label{curve_subsol}
 For an arcwise continuous function $u$ in $\X\times(0,T)$, the following conditions are equivalent:
    \begin{itemize}
        \item[(i)] $u$ is a curve-based viscosity subsolution of \eqref{eq time-ind} 
        \item[(ii)] For any $(x,t)\in\X\times(0,T)$ and any $\gamma\in \mathcal{A}_x(\X)$, 
        \[
        u(x,t)\leq \int_0^\delta L(\gamma(\sigma),|\gamma'|(\sigma))\,  d\sigma + u(\gamma(\delta),t-\delta)
        \]
        holds for all $\delta\in [0,t)$.
    \end{itemize}
\end{prop}

\begin{prop}[\cite{Na1}, Proposition 3.2]\label{curve_supersol}
	For an arcwise continuous function $u$ in $\X\times(0,T)$, the following conditions are equivalent:
    \begin{itemize}
        \item[(i)] $u$ is a curve-based viscosity supersolution of \eqref{eq time-ind} 
        \item[(ii)] For  any $(x,t)\in\X\times(0,T)$ and any $\varepsilon>0$, there exists $\gamma\in \mathcal{A}_x(\X)$ such that
        \[
        u(x,t)\geq \int_0^\delta L(\gamma(\sigma),|\gamma'|(\sigma))\,  d\sigma + u(\gamma(\delta),t-\delta)-\varepsilon
        \]
        holds for all $\delta\in [0,t)$.
    \end{itemize}
\end{prop}

Moreover, an arcwise continuous function $u$ in $\X\times(0,T)$ is said to be a curve-based viscosity solution of \eqref{eq time-ind} if $u$ is a curve-based subsolution and a curve-based supersolution of \eqref{eq time-ind}. In other words, $u$ is a curve-based viscosity solution of \eqref{eq time-ind} if and only if for any $(x, t)\in \X\times (0, T)$ and any $h\in [0, t)$, there holds
\begin{equation}\label{c-sol dpp}
u(x,t)=\inf_{\gamma\in \mathcal{A}_x(\X)}\left\{\int_0^\delta L(\gamma(\sigma),|\gamma'|(\sigma))d\sigma +u(\gamma(\delta), t-\delta)\right\}.
\end{equation}
Such a definition of solutions naturally leads us to the following solution formula for the associated initial value problem:
\begin{equation}\label{value time-ind}
u(x,t)=\inf_{\gamma\in \mathcal{A}_x(\X)}\left\{\int_0^t L(\gamma(\sigma),|\gamma'|(\sigma))d\sigma +u_0(\gamma(t))\right\}    \quad \text{for $(x, t)\in \X\times [0, T)$,}
\end{equation}
where $u_0\in C(\X)$ is the initial value as in \eqref{initial}. 
In fact, it is shown in \cite[Theorem~4.2]{Na1} that $u$ defined by \eqref{value time-ind} is the unique curve-based viscosity solution of \eqref{eq time-ind} and \eqref{initial} provided that $u_0$ is bounded uniformly continuous in $\X$ and $H\in C(\X\times [0, \infty))$ satisfies a set of assumptions similar to our conditions (H1)--(H3) for time-independent $H$. The assumptions imposed on $H$ in \cite{Na1} are actually weaker than (H1)--(H3), and the well-posedness result applies to both \eqref{eq special0} and the case of superlinear Hamiltonians. The uniqueness of solutions is a consequence of a comparison theorem given in \cite[Theorem~4.8]{Na1}. 

It is possible to study the more general equation \eqref{eq general} by extending \eqref{value time-ind} in a natural way to the case of Hamiltonians that depend also on the time variable. This is related to our discussions in Section~\ref{sec:eikonal-exist} and Section~\ref{sec:suplin-exist}.

\subsection{Slope-based viscosity solutions}
Let us next review the definition and some properties of slope-based solutions. We refer to \cite{AF, GaS2, GaS} for more details about the notion and related discussions.  
We first recall the definitions of subsolution and supersolution test functions for the slope-based solutions of \eqref{eq general}. 

\begin{defi}[\cite{GaS}]\label{def s-sol}
	A function \(\psi: \X\times(0,T)\to\mathbb{R}\) is called a subsolution test function if \(\psi(x,t)=\psi_1(x,t)+\psi_2(x,t)\) for any $(x, t)\in  \X\times (0, T)$, where  \(\psi_1,\psi_2\) are locally Lipschitz, \(|\nabla\psi_1|=|\nabla^-\psi_1|\) is continuous, and \(\partial_t\psi_1,\partial_t\psi_2\) are continuous in $\X\times (0, T)$. We denote by $\underline{\mathcal{C}}$ the class of subsolution test functions. 
    A function \(\psi:\X\times(0,T) \to\mathbb{R}\) is called a supersolution test function if $-\psi\in \underline{\mathcal{C}}$. We denote by $\overline{\mathcal{C}}$ the class of supersolution test functions.
\end{defi}
For a function $\psi:\X\times(0,T)\to\mathbb{R}$ we will write $\psi^*$ to denote its upper semicontinuous envelope, i.e.
\[
\psi^*(x,t)=\limsup_{(y,s)\to(x,t)}\psi(y,s),\quad (x,t)\in \X\times(0,T).
\]
We define 
\[  H(x, t, p):=H(x, t, 0),\qquad \text{for all}\ x\in\X, t\in (0, T), \ p<0.\]
For any $(x, t, p)\in \X\times (0, T)\times  \R$ and $a\geq 0$, denote
\begin{equation}\label{H-extension}
H_a(x, t, p) = \inf_{|q-p| \le a}H(x, t, q), \quad H^a(x, t, p) = \sup_{|q-p| \le a}H(x, t, q). 
\end{equation}
If $p\mapsto H(x, t, p)$ is nondecreasing, we  see that
\[
H_a(x,t,  p) = H(x,t, p-a), \quad H^a(x,t,  p) = H(x,t,  p+a).
\]
Next, we introduce the definition of slope-based viscosity solution of \eqref{eq general}.

\begin{defi}\label{def metric special}
	A locally bounded upper semicontinuous function \(u: \X\times(0,T) \to\mathbb{R}\) is a slope-based viscosity subsolution of \eqref{eq general} if  
	\[
		\partial_t\psi(x,t)+H_{|\nabla\psi_2|^*(x,t)}(x, t, |\nabla\psi_1|(x,t))\leq 0
    \]
	whenever \(u-\psi\) has a local maximum at \((x,t)\in\X\times(0,T)\) for \(\psi\in \underline{\mathcal{C}}\). 
	A locally bounded lower semicontinuous function \(u: \X\times(0,T)\to\mathbb{R}\) is a slope-based viscosity supersolution of \eqref{eq general} if  
    \[
	\partial_t\psi(x,t)+H^{|\nabla\psi_2|^*(x,t)}(x, t, |\nabla\psi_1|(x,t))\geq 0
    \]
	whenever \(u-\psi\) has a local minimum at \((x,t)\in\X\times(0,T)\) for \(\psi\in \overline{\mathcal{C}}\). Moreover, $u\in C(\X\times(0,T))$ is a slope-based viscosity solution of \eqref{eq general} if it is a slope-based viscosity subsolution and a slope-based viscosity supersolution of \eqref{eq general}. 
\end{defi}

Uniqueness and existence of slope-based solutions to the initial value problem \eqref{eq general}-\eqref{initial} is given in \cite{GaS}. 
In particular, comparison principles are established separately for sublinear and superlinear Hamiltonians under different sets of assumptions. For our purpose here, we include a comparison principle below for the eikonal-type equation \eqref{eq special}. It can be derived from \cite[Proposition 3.3]{GaS}, which applies to possibly unbounded slope-based solutions to general sublinear Hamilton-Jacobi equations.

\begin{thm}[Comparison principle for slope-based solutions to eikonal-type equation]\label{thm:comparison-eikonal}
    Let $(\X, d)$ be a complete length space and $T>0$. Assume that $f: \X\times (0, T)\to \R$ is uniformly continuous. Let $u$ be a bounded slope-based viscosity subsolution of \eqref{eq special} and $v$ be a bounded slope-based viscosity supersolution of \eqref{eq special}. If $u_0:\X\to \R$ is uniformly continuous and 
    \[
     \limsup_{t\to 0+}\sup_{x\in \X}(u(x, t)-u_0(x))\leq 0,\ \quad \liminf_{t\to 0+}\inf_{x\in \X}(v(x, t)-u_0(x))\geq 0,
    \]
    then $u\leq v$ in $\X\times [0, T)$. 
\end{thm}
Let us include another comparison result for the case of superlinear Hamiltonians independent of $t$, which is a simplified version of \cite[Theorem~4.2]{GaS}. 

\begin{thm}[Comparison principle for slope-based solutions for superlinear Hamiltonians]\label{unique_slope}
    Let $(\X, d)$ be a complete length space and $T>0$. Fix $x_0\in \X$. Let 
    \begin{equation}\label{special-H}
H(x,p)=\tilde{H}(x,p)-f(x),\qquad \text{for}\ x\in\X,\; p\in[0,\infty), 
\end{equation}
 where $\tilde{H}: \X\times \R\to \R$ is assumed to satisfy the following conditions. 
    \begin{itemize}
        \item[(1)] $\tilde{H}$ is uniformly continuous on bounded subsets of $\X\times [0, \infty)$, and for every $x\in\X$, $p\mapsto \tilde{H}(x, p)$ is convex and nondecreasing in $[0, \infty)$.
        \item[(2)] There exist $C_2\geq C_1>0,\;\alpha>1$ such that for every $(x,p)\in \X\times[0,\infty)$
        \begin{equation}\label{special-H-est1}
        C_1p^\alpha\leq \tilde{H}(x,p)\leq C_2p^\alpha.
        \end{equation}
        \item[(3)] There exist $h, \ol{h}: (0, \infty)\to \R$ and $s_0>1$ such that $h(s)>1$ for $1<s<s_0, \overline{h}(s)\to 1$ as $s\to 1$, and 
        \[ s h(s)\tilde{H}(x,p)\leq \tilde{H}(x, sp)\leq \overline{h}(s)\tilde{H}(x,p) \]
        for all $s>0$ and $(x,p)\in \X\times[0,\infty)$. 
        \item[(4)] There is a modulus of continuity $\tilde{\omega}$ such that 
        \begin{equation}\label{special-H-est2}
        |\tilde{H}(x,p)-\tilde{H}(y,p)|\leq \tilde{\omega}(d(x,y))(1+p^\alpha)
        \end{equation}
        holds for all  $(x,y,p)\in\X\times\X\times[0,\infty)$. 
    \end{itemize}
    Assume in addition that $u_0: \X\to \R$ is bounded Lipschitz and $f: \X\to \R$ is bounded uniformly continuous. Let $u$ be a bounded slope-based viscosity subsolution of \eqref{eq time-ind} and $v$ be a slope-based viscosity supersolution of \eqref{eq time-ind} satisfying
    \[
    \lim_{t\to 0}\sup_{\X}\{u(x,t)-v(x,t)\}\leq 0,
    \]
    then $u\leq v$ on $\X\times(0,T)$.
\end{thm}
In the original statement of \cite[Theorem~4.2]{GaS}, the upper bound of $\tilde{H}$ in \eqref{special-H-est1} and the uniform continuity in \eqref{special-H-est2} are allowed to depend also on the bounds of $d(x, x_0)$ and $d(x, y_0)$. The function $f$ is also assumed to be uniformly continuous only locally rather than in the entire $\X$. We strengthen all these assumptions here for simplicity and convenience in our later applications. 

With the initial condition \eqref{initial} interpreted as \eqref{initial-conti0}, 
we also have the existence of a unique slope-based viscosity solution of the initial value problem for \eqref{eq general}. The existence result for sublinear Hamiltonians is obtained by adapting Perron's method to the metric setting; see \cite[Theorem~3.4]{GaS}. For the case of superlinear Hamiltonians, an optimal control interpretation is used to construct the unique slope-based solution to the initial value problem. In fact, under appropriate assumptions on $H$ and  the associated Lagrangian $L$ given by \eqref{legendre0}, the slope-based solution is still represented by the formula \eqref{value time-ind} with $\mathcal{A}_x(\X)$ replaced by the class of absolutely continuous functions $\gamma$ on $[0, \infty)$ satisfying $\gamma(0)=x$. The assumptions on $H$ and $L$ are closely related to our conditions (H1)--(H4) and cover the typical example considered in Theorem~\ref{unique_slope}. We omit the detailed description of these assumptions and refer the reader to \cite[Section~4.2]{GaS}. See also \cite[Section~3.3]{AF} for the control-theoretic approach under different assumptions on $H$.




\section{The case of eikonal-type equations}\label{sec:eikonal}


\subsection{Definition of Monge solutions}\label{sec:eikonal-monge-def}

Let $T>0$. Let us study \eqref{eq special} under the assumption that $(\X, d)$ is a complete length space. Below we denote $\Y_T=\X\times (0, T)$ for simplicity of notation. For our new notion of solutions, we consider the following function classes. 
For $k\geq 0$, let $S_k(Y_T)$ be defined as in \eqref{sk}. 
We also take $S(\Y_T)=\bigcup_{k\geq 0} S_k(\Y_T)$. 
It is clear that $S_0(\Y_T)$ denotes the set of all functions $u\in \lipl(\Y_T)$ that are nondecreasing in $t$. Moreover, any function $u\in \lipl(\Y_T)$ belongs to $S_k(\Y_T)$ for $k\geq 0$ provided that 
\[
|u(x, t)-u(x, s)|\leq k|t-s| \quad \text{for all $x\in \X$ and $t, s\in [0, T)$}.    
\]
In particular, we have $\lip(\Y_T)\subset S(\Y_T)$. 

Suppose that $u\in S_k(\Y_T)$ with $k\geq 0$ satisfying 
\begin{equation}\label{rhs-positive}
k+\inf_{\Y_T} f\geq 0.    
\end{equation}
Letting $v$ be as in \eqref{unknown-change}, 
we see that formally $v$ should solve \eqref{eq special2}, whose right hand side is 
nonnegative thanks to \eqref{rhs-positive}. Let us now introduce the definition of Monge solutions of \eqref{eq special}. As mentioned in the introduction, our notion of Monge solutions is based on the subslope $|D^-v|$ of $v$ in space-time as defined in \eqref{def D-0}. 
By the monotonicity of $v$ in time, it is not difficult to see that $|D^-v|$ in \eqref{def D-0} is equivalent to the following form:
\begin{equation}\label{def D-}
|D^- v|(x, t)=\limsup_{\delta \to 0+}\sup\left\{\frac{v(x, t)-v(y, t-\delta)}{\delta}: y\in \X,\ d(x, y)\leq \delta\right\}.     
\end{equation}
Let us use this simple equivalent expression in our definition below.

\begin{defi}\label{def monge special}
A function $u\in S(\Y_T)$ is said to be a Monge solution (resp., Monge subsolution, Monge supersolution) of \eqref{eq special} if there exists $k\geq 0$ satisfying \eqref{rhs-positive} such that $u\in S_k(\Y_T)$, and $v\in S_0(\Y_T)$ given by \eqref{unknown-change} satisfies  
\begin{equation}\label{subslope-1}
 |D^-v|(x, t)= f(x, t)+k\quad \text{(resp., $\leq$, $\geq$)}    
\end{equation}
for any $(x, t)\in \Y_T$, where $|D^-v|(x, t)$ is given as in \eqref{def D-}. 
\end{defi}
Note that when $u\in S_k(\Y_T)$ is a Monge solution (resp., Monge subsolution, Monge supersolution) as defined above, the function $v$ in \eqref{unknown-change} is a Monge solution (resp., Monge subsolution, Monge supersolution) of \eqref{eq special2}. 
In fact, for the same $k\geq 0$, it is easily seen that $v\in S_0(\Y_T)$ and $v$ itself already satisfies \eqref{subslope-1}.

\begin{rmk}\label{rem1}
    It is worth pointing out that, since $v\in S_0(\Y_T)$,  we have $|D^- v|(x, t)\geq 0$ at each $(x, t)\in \Y_T$ by considering the limit in \eqref{def D-} via $y=x$ as $\delta\to 0+$.     
    It follows that, for all  $(x, t)\in \Y_T$, we have 
    \begin{equation}\label{def D-2}
    |D^- v|(x, t)=\limsup_{\delta \to 0+}\sup\left\{\frac{v(x, t)-v(y, s)}{\delta}: (y,s)\in \Y_T, \ 0\leq t-s\leq \delta,\ d(x, y)\leq \delta\right\}     
    \end{equation}
    as well as
    \begin{equation}\label{def D-3}
    |D^- v|(x, t)= \limsup_{\delta \to 0+}\sup\left\{\frac{\max\{v(x, t)-v(y, s), 0\}}{\delta}: (y,s)\in \Y_T, \ 0\leq t-s,\ d(x, y)\leq \delta\right\}. 
    \end{equation}
    Moreover, one can relax the condition 
    \begin{equation}\label{D-cond1}
    0\leq t-s\leq \delta,\quad  d(x, y)\leq \delta    
    \end{equation}
    for the supremum in \eqref{def D-2} and \eqref{def D-3} to
    \begin{equation}\label{D-cond2}
     |t-s|\leq \delta,\quad  d(x, y)\leq \delta
    \end{equation}
    for $\delta<t$. To see this, letting $Q_v(x, t, \delta)=(v(x, t)-v(y, s))/ \delta$,
    we have
    \begin{equation}\label{D-cond3}
    \sup\left\{Q_v(x, t, \delta): \text{$(y, s)$ satisfies \eqref{D-cond1}}\right\}\leq \sup\left\{Q_v(x, t, \delta): \text{$(y, s)$ satisfies \eqref{D-cond2}}\right\}.    
    \end{equation}
On the other hand, if $(y, s)$ fulfills \eqref{D-cond2} with $s>t$, then 
by the monotonicity of $v$ in time, we have
\[
v(x, t)-v(y,t)\geq  v(x, t)-v(y, s), 
\]
which yields 
\begin{equation}\label{D-cond4}
 \sup\left\{Q_v(x, t, \delta): \text{$(y, s)$ satisfies \eqref{D-cond1}}\right\}\geq \sup\left\{Q_v(x, t, \delta): \text{$(y, s)$ satisfies \eqref{D-cond2}}\right\}.
\end{equation}
Combining \eqref{D-cond3} and \eqref{D-cond4}, we obtain 
\[
\sup\left\{Q_v(x, t, \delta): \text{$(y, s)$ satisfies \eqref{D-cond1}}\right\}= \sup\left\{Q_v(x, t, \delta): \text{$(y, s)$ satisfies \eqref{D-cond2}}\right\}.
\]
Hence, we are led to 
\[
|D^- v|(x, t)=\limsup_{\delta \to 0+}\sup\left\{\frac{v(x, t)-v(y, s)}{\delta}: (y,s)\in \Y_T, \ |t-s|,\; d(x, y)\leq \delta\right\}.     
\]
By a similar argument, we can replace the condition \eqref{D-cond1} in \eqref{def D-3} by \eqref{D-cond2}. Hence, our expressions of $|D^-v|$ in \eqref{def D-}, \eqref{def D-2} and \eqref{def D-3} are all equivalent to that introduced in \eqref{def D-0}. 
\end{rmk}

In Definition \ref{def monge special}, the conditions required on $u$ involve the constant $k\geq 0$. Our definition, however, does not depend on the choice of any particular $k$. The following result clarifies this important issue. 

\begin{lem}\label{ind_of_k}
Let $(\X,d)$ be a complete length space and $T>0$. Let $\Y_T=\X\times (0, T)$. Assume that $f: \Y_T\to \R$ is bounded below. If $u\in S(\Y_T)$ is a Monge solution (resp., Monge subsolution, Monge supersolution) of \eqref{eq special}, then $v$ in \eqref{unknown-change} satisfies \eqref{subslope-1} for any $k\geq 0$ such that $u\in S_k(\Y_T)$ and \eqref{rhs-positive} hold. 
\end{lem}

\begin{proof}
Suppose that $u$ is a Monge solution of \eqref{eq special}. Then, there exists $k_0\geq 0$ such that  $\inf_{\Y_T} f\geq -k_0$ and $v_0(x,t)=u(x,t)+k_0t$ satisfies
    \begin{equation}\label{eq_v_1}
        |D^-v_0|(x,t)=f(x,t)+k_0
    \end{equation}
    for all $(x, t)\in \Y_T$. Fix arbitrarily $k\geq 0$ such that $u\in S_{k}(\Y_T)$ and \eqref{rhs-positive} hold. Let $v(x,t)=u(x,t)+kt$. It is easily seen that $v\in S_0(\Y_T)$. For any $\delta>0$, we have
    \[
    \sup_{d(y, x)\leq \delta}\dfrac{v(x,t)-v(y,t-\delta)}{\delta}=\sup_{d(y, x)\leq \delta} \dfrac{v_0(x,t)-v_0(y,t-\delta)}{\delta}+(k-k_0).
    \]
    Owing to \eqref{eq_v_1}, we pass to the limsup as $\delta\to 0+$ to obtain
    \[
    |D^-v|(x,t)=|D^-v_0|(x,t)+(k-k_0)=f(x,t)+k.
    \]
One can similarly prove the inequalities corresponding to the Monge subsolution and Monge supersolution properties.  
\end{proof}


\subsection{Comparison principle}

We now study the well-posedness of the initial value problem for \eqref{eq special} with the initial value condition \eqref{initial}. Let us begin with a comparison principle for \eqref{eq special}. 

\begin{thm}[Comparison principle]\label{thm:comparison1}
    Let $(\X,d)$ be a complete length space and $T>0$. Let $\Y_T=\X\times (0, T)$. Assume that $f: \Y_T\to \R$ is bounded below.   
    Assume that $u_1, u_2\in S(\Y_T)$ are respectively a bounded Monge subsolution and a bounded Monge supersolution of \eqref{eq special}. If 
    \begin{equation}\label{comp1}
       \lim_{t\to 0+}\sup_{x\in \X}(u_1(x,t)-u_2(x,t))\leq 0, 
    \end{equation}
    then $u_1\leq u_2$ in $\Y_T$. 
\end{thm}
\begin{proof}
As required in the definition of Monge solutions, we can find $k\geq 0$ large enough such that $u_1, u_2\in S_k(\Y_T)$ and \eqref{rhs-positive} holds. Let  
\begin{equation}\label{unknown-change-i}
v_i(x, t)=u_i(x, t)+kt \quad \text{for $(x,t)\in \Y_T$ and $i=1, 2$.}
\end{equation}
In view of \eqref{comp1}, it is clear that under this change
\begin{equation}\label{comp1-new}
       \lim_{t\to 0}\sup_{\X}\{v_1(x,t)-v_2(x,t)\}\leq 0.
    \end{equation}
Moreover, by Definition \ref{def monge special} and Lemma \ref{ind_of_k}, we have
\begin{equation}\label{comp2-new}
|D^- v_1|(x, t)\leq f(x, t)+k\leq |D^-v_2|(x, t),\quad \text{for all $(x, t)\in \Y_T$.}
\end{equation}
 
It then suffices to show that $v_1\leq v_2$ in $\Y_T$. Assume by contradiction that $\sup_{\Y_T}\ (v_1-v_2)>\mu$ for some $\mu>0$. Let $\kappa>0$ and we define 
    \[
    \Phi(x,t):= v_1(x,t)-v_2(x,t)-\dfrac{\kappa}{T-t},\quad  (x,t)\in \X\times(0,T).
    \]
    For sufficiently small $\kappa>0$, we have $\sup_{\Y_T}\Phi>\mu$. By \eqref{comp1-new}, there exists $0<\eta<T$ sufficiently small such that 
    \begin{equation}\label{th1}
       \sup_{\X\times (0, 2\eta]} \Phi\leq \sup_{\X\times (0, 2\eta]} (v_1-v_2)<\frac{\mu}{2}.
    \end{equation}
    In addition, since $v_1, v_2$ are bounded, we get
    \begin{equation}\label{th2}
    \sup_{\X\times [T-2\eta, T)} \Phi <\frac{\mu}{2} 
    \end{equation}
    by letting $\eta>0$ smaller if necessary. It follows that
    \[
    \sup_{\X\times [2\eta, T-2\eta]}\Phi=\sup_{\Y_T}\Phi >\mu.
    \] 
    Then there exists $z_0\in \X\times (2\eta, T-2\eta)$ such that 
    $\Phi(z_0)\geq \sup_{\X\times [\eta, T-\eta]}\Phi- \varepsilon^2>\mu$
    for $\vep\in (0, \eta)$ sufficiently small. In particular, we can choose $\vep\in (0, \eta)$ satisfying
    \begin{equation}\label{cont1}
        \varepsilon<\dfrac{\kappa}{T^2}.
    \end{equation}    
Since $(\X,d)$ is a complete length space, we see that $(\X\times [\eta, T-\eta],\bar{d})$ is also a complete metric space, where $\bar{d}$ the space-time metric given by \eqref{metric spacetime}. By Ekeland's variational principle (cf. \cite{Eke1, Eke2}), there exists $z_\vep=(x_{\varepsilon},t_{\varepsilon})\in  \overline{B_{\varepsilon}(z_0)}\subset \X\times (\eta, T-\eta)$ such that $z\mapsto \Phi(z)-\varepsilon \bar{d}(z, z_\vep)$ attains a local maximum in $\X\times[\eta, T-\eta]$ at $z=z_\vep$. Here, $B_r(z)$ denotes the open ball in $(\Y_T, \bar{d})$ centered at $z$ with radius $r>0$.

It then follows that, for all $z=(x, t)$ near $z_\vep$, 
    \[
        v_2(z_\vep)-v_2(z)\leq  v_1(z_\vep)-v_1(z) +\dfrac{\kappa}{T-t}-\dfrac{\kappa}{T-t_\varepsilon}+\varepsilon \bar{d}(z_\vep, z).
    \]
In particular, for any $\delta>0$ small,  we have 
\[
v_2(x_\vep, t_\vep)-v_2(x, t_\vep-\delta)\leq  v_1(z_\vep)-v_1(x, t_\vep-\delta) +\dfrac{\kappa}{T-t_\vep+\delta}-\dfrac{\kappa}{T-t_\varepsilon}+\varepsilon \delta
\]
for all $x\in \X$ satisfying $d(x, x_\vep)\leq \delta$. We thus obtain 
    \[
    \begin{aligned}
    &\sup \left\{\frac{v_2(x_\vep, t_\vep)-v_2(x, t_\vep-\delta)}{\delta}: d(x_\vep, x)\leq \delta\right\}\\
    &\leq \sup \left\{\frac{v_1(z_\vep)-v_1(x,t_\vep-\delta)}{\delta}: d(x_\vep, x)\leq \delta\right\} + \frac{1}{\delta}\left(\dfrac{\kappa}{T-t_\vep+\delta}-\dfrac{\kappa}{T-t_\varepsilon}\right) +\varepsilon     
    \end{aligned}
    \]
   for $\delta>0$ small.  
    Taking the limsup as $\delta\to 0+$, we get
    \[
    |D^- v_2|(z_\vep)\leq |D^- v_1|(z_\vep)-\frac{\kappa}{(T-t_\vep)^2}+\vep\leq |D^- v_1|(z_\vep)-\frac{\kappa}{T^2}+\vep.
    \]
Then, \eqref{comp2-new} further yields $\kappa/T^2 \leq \vep$, which is a contradiction to \eqref{cont1}. 
\end{proof}

As an immediate consequence of the comparison result, we obtain the uniqueness of bounded Monge solutions $u$ of the initial value problem for \eqref{eq special} with the initial condition \eqref{initial} interpreted as \eqref{initial-conti0}. Note that Theorem~\ref{thm:comparison1} above does not require $f$ to be continuous. This suggests a potential application to the uniqueness of solutions even for equations with a discontinuous inhomogeneous term $f$. See \cite{LShZ2} for recent developments for the stationary case.


\subsection{Existence of Monge solutions}\label{sec:eikonal-exist}
Let us now consider the existence of Monge solutions of the initial value problem and prove Theorem~\ref{thm:into1}. 
Our existence result is a straightforward adaptation of the classical argument based on a control-based formula. 

Consider the value function $u$, defined by \eqref{eq value_f0}, of an optimal control problem, where the curve class $\Gamma_x$ is given by \eqref{eikonal-curve} for $x\in \X$. 
It is easy to see that $u(x,0)=u_0(x)$ for all $x\in\X$. We show that $u$ given by \eqref{eq value_f0} is a Monge solution of \eqref{eq special}, assuming that $f$ is bounded in $\Y_T$ and is  Lipschitz either in space or in time; that is, there exists $L_f>0$ such that either
 \begin{equation}\label{lip-f-sp}
     |f(x,t)-f(y,t)|\leq L_fd(x,y) \qquad \text{for all}\ x,y\in \X,\; t\in(0,T) 
 \end{equation}
 or 
 \begin{equation}\label{lip-f-time}
     |f(x,t)-f(x,s)|\leq L_f|t-s| \qquad \text{for all}\ x\in \X,\; t,s\in(0,T). 
 \end{equation}

We first show the boundedness and Lipschitz continuity of $u$ in $\X\times [0, T)$. 
\begin{prop}[Boundedness and Lipschitz regularity]\label{prop bound-lip}
    Let $(\X,d)$ be a complete length space. Let $T>0$. Assume that $u_0\in \lip(\X)$ is bounded. Assume also that $f$ is bounded on $\Y_T$, and either \eqref{lip-f-sp} or \eqref{lip-f-time} holds for some $L_f>0$. Then, $u$ given by \eqref{eq value_f0} is bounded and belongs to $\lip(\X\times [0, T))$. In particular, there exists $K>0$ such that 
    \begin{equation}\label{initial-conti}
     \sup_{x\in \X} |u(x, t)-u_0(x)|\leq K t \quad \text{for all $t\in [0, T)$.}
     \end{equation}
\end{prop}
\begin{proof}
    The boundedness of $u$ in $\X\times [0, T)$ is an immediate consequence of the boundedness of $f$ and $u_0$. In fact, by the formula of $u$ as in \eqref{eq value_f0}, we have 
\[
\sup_{\Y_T}|u| \leq \sup_{\X} |u_0|+ T \sup_{\Y_T}|f|.
\]
Let us next prove that $u\in\lip(\X\times [0, T))$. Fix  $z_1=(x_1, t_1), z_2=(x_2, t_2)\in \X\times [0, T)$ arbitrarily. By \eqref{eq value_f0}, for every $\delta>0$, we can take a curve $\gamma_1\in \Gamma_{x_1}$  such that 
\begin{equation}\label{eq exist-eikonal5}
 u(z_1)\geq u_0(\gamma_1(t_1))+\int_0^{t_1}f(\gamma_1(\sigma),t_1-\sigma)\;d\sigma-\delta. 
\end{equation}
We can also find another arc-length parametrized curve $\xi: [0, \tau]\to \X$ with $\tau\leq d(x_1, x_2)+\delta$ such that $\xi(0)=x_2$, $\xi(\tau)=x_1$.   

Next, we build a curve $\gamma_2\in \Gamma_{x_2}$ as below: 
\[
\gamma_2(\sigma)=\begin{dcases}
\xi(\sigma) & \text{if $0\leq \sigma\leq \tau$,}\\
\gamma_1(\sigma-\tau) & \text{if $\sigma>\tau$.}
\end{dcases}
\]
It is not difficult to see that 
\begin{equation}\label{eq exist-eikonal2}
d(\gamma_2(t_2), \gamma_1(t_1))
\leq |t_1-t_2|+\tau\leq |t_1-t_2|+d(x_1, x_2)+\delta.
\end{equation}
Below we give an estimate of 
\[
F(\gamma_1,\gamma_2,t_1,t_2):=\int_0^{t_1} f(\gamma_1(\sigma), t_1-\sigma)\, d\sigma -\int_0^{t_2} f(\gamma_2(\sigma), t_2-\sigma)\, d\sigma.
\]
If $t_2\leq \tau$, we have
\begin{equation}\label{eq lip-bound1}
    |F(\gamma_1,\gamma_2,t_1,t_2)|\leq (t_1+\tau)\sup_{\Y_T}|f|\leq 2\sup_{\Y_T}|f|(|t_1-t_2|+\tau).
\end{equation}
Let us now consider the case $t_2>\tau$. To this end, we first estimate $F(\gamma_1, \gamma_1, a, b)$ for any $a, b\in [0, T)$ with $a\geq b$. It is not difficult to see that 
\begin{equation}\label{eq lip-bound2-1}
    \begin{aligned}
        |F(\gamma_1, \gamma_1, a, b)| \leq &\ \left|\int_0^{a-b} f(\gamma_1(\sigma), a-\sigma)\, d\sigma\right|\\
        &+\left|\int_0^{b} f(\gamma_1(a-b+\sigma), b-\sigma)-f(\gamma_1(\sigma), b-\sigma)\, d\sigma\right| 
    \end{aligned}
\end{equation}
and 
\begin{equation}\label{eq lip-bound2-2}
        |F(\gamma_1, \gamma_1, a, b)|  \leq  \left|\int_{b}^{a} f(\gamma_1(\sigma), a-\sigma)\, d\sigma\right|+\left|\int_0^{b} f(\gamma_1(\sigma), a-\sigma)-f(\gamma_1(\sigma), b-\sigma)\, d\sigma\right|. 
\end{equation}
If the space-Lipschitz condition \eqref{lip-f-sp} holds, we apply it to \eqref{eq lip-bound2-1}. If the time-Lipschitz condition \eqref{lip-f-time} holds, we use it for \eqref{eq lip-bound2-2}. In either case, we obtain
\[
|F(\gamma_1, \gamma_1, a, b)|  \leq \ \left(\sup_{\Y_T}|f|+bL_f\right)(a-b)\leq \left(\sup_{\Y_T}|f|+TL_f\right)(a-b).
\]
For the case $b\geq a$, we can apply the same bound with $a, b$ exchanged. Hence, for all $a, b\in [0, T)$,   we have
\begin{equation}\label{eq lip-bound2-3}
    |F(\gamma_1, \gamma_1, a, b)|  \leq \left(\sup_{\Y_T}|f|+TL_f\right)|a-b|.
\end{equation}
We now use \eqref{eq lip-bound2-3} to estimate $|F(\gamma_1,\gamma_2,t_1,t_2)|$ in the case $t_2>\tau$. Note first that
\begin{align*}
 &|F(\gamma_1,\gamma_2,t_1,t_2)|\leq \ \left|\int_0^{\tau} f(\gamma_1(\sigma), t_1-\sigma)\, d\sigma\right|\\
 &\qquad +\left|\int_0^{t_1} f(\gamma_1(\sigma), t_1-\sigma)\, d\sigma-\int_0^{t_2-\tau} f(\gamma_1(\sigma), t_2-\tau-\sigma)\, d\sigma\right|=:A_1+A_2.   
\end{align*}
It is easily seen that $A_1\leq \tau\sup_{\Y_T}|f|$. For $A_2$, we apply \eqref{eq lip-bound2-3} with $a=t_1, b=t_2-\tau$ to get
\[
A_2=|F(\gamma_1, \gamma_1, t_1, t_2-\tau)|  \leq \left(\sup_{\Y_T}|f|+TL_f\right)|t_1-t_2+\tau|. 
\]
Therefore, combining the bounds for $A_1$ and $A_2$, we get
\begin{equation}\label{eq lip-bound2}
    |F(\gamma_1,\gamma_2,t_1,t_2)|\leq \left(2\sup_{\Y_T}|f|+TL_f\right)(|t_1-t_2|+\tau).
\end{equation}
In summary, by combining \eqref{eq lip-bound1} and \eqref{eq lip-bound2}, 
we are led to 
\begin{equation}\label{eq exist-eikonal3}
|F(\gamma_1,\gamma_2,t_1,t_2)|\leq \left(2\sup_{\Y_T}|f|+TL_f\right)(|t_1-t_2|+d(x_1, x_2)+\delta).
\end{equation}
Since $u_0$ is Lipschitz continuous in $\X$,  
by \eqref{eq exist-eikonal2} we get
\begin{equation}\label{eq exist-eikonal4}
u_0(\gamma_2(t_2))-u_0(\gamma_1(t_1))\leq C_0 (|t_1-t_2|+d(x_1, x_2)+\delta)
\end{equation}
for some $C_0>0$. 
Noticing that by definition 
\[
u(z_2)\leq u_0(\gamma_2(t_2))+\int_0^{t_2} f(\gamma_2(\sigma), t_2-\sigma)\, d\sigma,
\]
we then adopt \eqref{eq exist-eikonal5}, \eqref{eq exist-eikonal3}, and \eqref{eq exist-eikonal4} to deduce 
\[
\begin{aligned}
&u(z_2)-u(z_1)-\delta\\
&\leq u_0(\gamma_2(t_2))-u_0(\gamma_1(t_1))+ \left|\int_0^{t_1} f(\gamma_1(\sigma), t_1-\sigma)\, d\sigma -\int_0^{t_2} f(\gamma_2(\sigma), t_2-\sigma)\, d\sigma\right|\\
&\leq \left(C_0+2\sup_{\Y_T}|f|+TL_f\right) (|t_1-t_2|+d(x_1, x_2)+\delta). 
\end{aligned}
\] 
Letting $\delta\to 0$, we deduce that 
\[
u(z_2)-u(z_1)\leq K(|t_1-t_2|+d(x_1, x_2)) 
\]
for some $K>0$ independent of $z_1=(x_1, t_1), z_2=(x_2, t_2)$. We conclude our proof of the Lipschitz continuity of $u$ by exchanging the roles of $z_1, z_2\in \X\times [0, T)$. 
\end{proof}

 It is well known in the Euclidean case that the function $u$ defined by \eqref{eq value_f0} satisfies the so-called dynamic programming principle. We extend the standard argument to length spaces. 

\begin{lem}[Dynamic programming principle]\label{lem dpp1}
    Let $T>0$ and $(\X,d)$ be a complete length space. Assume that $f$ is bounded on $\Y_T$. Let $u:\X\times [0,T)\to \R$ be defined by \eqref{eq value_f0}. Then, at any $(x,t)\in \Y_T$,
    \begin{equation}\label{eq dpp_0}
       u(x,t)=\inf_{\gamma\in \Gamma_{x}}\left\{ 
u(\gamma(\delta),t-\delta)+\int_0^{\delta} f(\gamma(\sigma),t-\sigma)\;d\sigma \right\} 
    \end{equation}
    holds for any $0<\delta\leq t$.
\end{lem}

\begin{proof}
    Fix $(x,t)\in \Y_T$ and let $0<\delta<t$. For every $\gamma\in \Gamma_{x}$, we take a portion $\gamma^\delta\in \Gamma_{\gamma(\delta)}$ of $\gamma$ defined by  $\gamma^{\delta}(\sigma)=\gamma(\delta+\sigma)$ for $\sigma\geq 0$ such that 
    \[
    \begin{aligned}
        \int_0^t f(\gamma(\sigma),t-\sigma)\, d\sigma +u_0(\gamma(t))= & \int_0^{\delta} f(\gamma(\sigma),t-\sigma)d\sigma\nonumber\\
        &\ +\int_0^{t-\delta} f(\gamma^{\delta}(\sigma),t-\delta-\sigma)\, d\sigma+ u_0(\gamma^{\delta}(t-\delta)).
    \end{aligned}
    \]
    Keeping the value of $\gamma$ in $[0,\delta]$ and taking the infimum of the right hand side over curves in $\Gamma_{\gamma(\delta)}$, we obtain
    $$
    \int_0^t f(\gamma(\sigma),t-\sigma)\;d\sigma +u_0(\gamma(t))\geq \int_0^{\delta}f(\gamma(\sigma),t-\sigma)\;d\sigma +u(\gamma(\delta),t-\delta).
    $$
    Next, taking the infimum for $\gamma\in \Gamma_{x}$ of both sides, we have
    $$
    u(x,t)\geq \inf_{\gamma\in \Gamma_{x}}\left\{ u(\gamma(\delta),t-\delta)+\int_0^{\delta} f(\gamma(\sigma),t-\sigma)\;d\sigma
 \right\}.
    $$
    On the other hand, for every $\gamma\in \Gamma_{x}$ and $\varepsilon>0$, take $\gamma_\vep\in \Gamma_{\gamma(\delta)}$ such that
    $$
    \int_0^{t-\delta} f(\gamma_\vep(\sigma),t-\delta-\sigma)\;d\sigma + u_0(\gamma_\vep(t-\delta))\leq u(\gamma(\delta),t-\delta)+\varepsilon.
    $$
    Then, we obtain that
    \[
    \begin{aligned}
        u(x,t)\leq& \int_0^{\delta}f(\gamma(\sigma),t-\sigma)\;d\sigma 
         +\int_0^{t-\delta}f(\gamma_\vep(\sigma),t-\delta-\sigma)\;d\sigma + u_0(\gamma_\vep(t-\delta))\\
        \leq& \int_0^{\delta}f(\gamma(\sigma),t-\sigma)\;d\sigma +u(\gamma(\delta),t-\delta)+\varepsilon.
    \end{aligned}
    \]
    Since $\gamma\in \Gamma_{x}$ and $\vep>0$ are arbitrary, it follows that
\[
    u(x,t)\leq \inf_{\gamma\in \Gamma_{x}}\left\{ u(\gamma(\delta),t-\delta)+\int_0^{\delta}f(\gamma(\sigma),t-\sigma)\;d\sigma
 \right\}.
\]
Our proof is thus complete. 
\end{proof}

We now verify the Monge solution property of the value function $u$. 
\begin{prop}[Monge solution property]\label{prop existence_2}
	Let $(\X,d)$ be a complete length space. Let $T>0$. Assume that $u_0\in \lip(\X)$ is bounded. Assume also that $f\in C(\Y_T)$ is bounded. If either \eqref{lip-f-sp} or \eqref{lip-f-time} holds for some $L_f>0$, then $u$ defined as in \eqref{eq value_f0} is a Monge solution of \eqref{eq special}.    
\end{prop}

\begin{proof}
We have shown in Proposition \ref{prop bound-lip} that $u\in \lip(\X\times [0, T))$. It follows that $u\in S_k(\Y_T)$ for some $k>0$. Moreover, we can take $k>0$ large enough such that \eqref{rhs-positive} also holds. Let $v$ be given by \eqref{unknown-change} for this $k>0$. Note that $v\in \lip(\X\times [0, T))\cap S_0(Y_T)$ and by \eqref{eq dpp_0}, $v$ satisfies
  \begin{equation}\label{eq dpp_1}
       v(x,t)=\inf_{\gamma\in \Gamma_{x}}\left\{ 
v(\gamma(\delta),t-\delta)+\int_0^{\delta} f(\gamma(\sigma),t-\sigma)+k\;d\sigma \right\} 
    \end{equation}
 for any $(x,t)\in \Y_T$ and $0<\delta\leq t$. Our goal is then to show that $|D^-v|=f+k$ holds everywhere in $\Y_T$. 

({Subsolution}) Let us first prove $|D^-v|\leq f+k$ in $\Y_T$. Fix arbitrarily $z=(x, t)\in \Y_T$. We take $0<\delta<t/2$ and $y\in \X$ arbitrarily such that $d(x, y)\leq \delta$ and $v(y, t-\delta)\leq v(x, t)$. To use the definition of $|D^- v|(x,t)$ to estimate its upper bound, we only need to consider points $(y, t-\delta)$ satisfying $v(y, t-\delta)\leq v(x, t)$, since we already know that $|D^- v|(x, t)\geq 0$, as emphasized in Remark \ref{rem1}. 

For any integer $j\geq 1$, there exists  
$\gamma_j\in \Gamma_{x}$ such that $\gamma_j(\delta_j)=y$ and
\[
\delta\leq
\ell(\gamma_j|_{[0,\delta_j]})=\delta_j\leq \delta\left(1+\frac{1}{j}\right).
\]
Here and in the sequel, we use $\ell(\gamma)$ to denote the length of a rectifiable curve $\gamma$. 
By \eqref{eq dpp_1}, we have 
\[
    v(x,t)\leq v(y,t-\delta_j)+\int_0^{\delta_j}\left(f(\gamma_j(\sigma), t-\sigma)+k\right)\, d\sigma.
\]
Thanks to the monotonicity of $t\mapsto v(y, t)$, we have $v(y, t-\delta_j)\leq v(y, t-\delta)$, which yields   
\[
0\leq v(x,t)-v(y,t-\delta)\leq v(x,t)-v(y,t-\delta_j). 
\]
It follows that
\begin{equation}\label{eik-exist eq1}
\begin{aligned}
    \dfrac{v(x,t)-v(y,t-\delta)}{\delta}&\leq \left(1+\frac{1}{j}\right)\dfrac{v(x,t)-v(y,t-\delta_j)}{\delta_j}\\
    &\leq \left(1+\frac{1}{j}\right)\dfrac{1}{\delta_j}\int_0^{\delta_j}\left(f(\gamma_j(\sigma),t-\sigma)+k\right) \, d\sigma. 
\end{aligned}
\end{equation}
By the continuity of $f$, for any $\vep>0$, we have 
\[
\sup_{\sigma\in [0, \delta_j]}|f(\gamma_j(\sigma),t-\sigma)-f(x, t)|<\vep
\]
for all $j\geq 1$ large and $\delta>0$ small. Then, \eqref{eik-exist eq1} yields
\[
\begin{aligned}
\dfrac{v(x,t)-v(y,t-\delta)}{\delta}& \leq \left(1+\frac{1}{j}\right)(f(x,t)+k+\vep). 
\end{aligned}
\]
Letting $j\to \infty$ and then $\delta\to 0+$, we obtain $|D^-v|(x,t)\leq f(x, t)+k+\vep$. 
Thanks to the arbitrariness of $\vep$, we complete the proof of the Monge subsolution property of $v$.     

\noindent ({Supersolution}) We next show $|D^-v|\geq f+k$ in $\Y_T$. Fix $z=(x,t)\in\Y_T$ and $\delta>0$ arbitrarily small. By \eqref{eq dpp_1}, there exists $\gamma_\delta\in \Gamma_{x}$ such that ${z}_\delta=(\gamma_\delta(\delta), t-\delta)\in \Y_T$ satisfies
\[
    v(z)+\delta^2\geq v({z}_\delta)+\int_0^{\delta} \left(f(\gamma_\delta(\sigma),t-\sigma)+k\right)\, d\sigma
\]
    which implies
    \begin{equation}\label{eq exist-eikonal1}
    \dfrac{1}{\delta}\left({v(z)-v({z}_\delta)}\right)\geq \dfrac{1}{\delta}\int_0^{\delta} f(\gamma_\delta(\sigma),t-\sigma)\, d\sigma+k-\delta.    
    \end{equation}
    Noticing that $\gamma_\delta\in \Gamma_{x}$ satisfies $|\gamma_\delta'|\leq 1$ a.e. in $(0, \delta)$, we have $d(\gamma_\delta(\delta), x)\leq \delta$. By the continuity of $f$, it then follows from \eqref{eq exist-eikonal1} that
    \[
         |D^-v|(x,t)\geq  \limsup_{\delta\to 0}\dfrac{1}{\delta}\int_0^{\delta}f(\gamma_{\delta}(\sigma),t-\sigma)\;d\sigma+k = f(x,t)+k, 
    \]
as desired. 
\end{proof}

Let us now complete our proof of Theorem~\ref{thm:into1}.
\begin{proof}[Proof of Theorem~\ref{thm:into1}]
We have shown in Proposition~\ref{prop bound-lip} that $u$ defined by \eqref{eq value_f0} is bounded and Lipschitz in $\X\times [0, T)$. It follows that there exist $k\geq 0$ such that $u\in S_k(\Y_T)$. By Proposition~\ref{prop existence_2}, we see that $u$ is a Monge solution of \eqref{eq special}. The initial condition \eqref{initial} also holds in the sense of \eqref{initial-conti0}, as a consequence of \eqref{initial-conti} in Proposition~\ref{prop bound-lip}. The uniqueness of such Monge solutions follows from Theorem~\ref{thm:comparison1}. 
\end{proof}

\begin{rmk}\label{special_cs_Monge}
    Without affecting the existence result, one can replace the curve class $\Gamma_{x}$ in \eqref{eq value_f0} by 
    \[
    \A^\ast_x:=\{\gamma\in\A_x(\X)\; |\; |\gamma'|\leq 1\ \text{a.e. in}\ [0,\infty)  \}, 
    \]
    where $\A_x(\X)$ is given in \eqref{admissible_set_gen}. Under the same assumptions as in Theorem~\ref{thm:into1}, we can use the same arguments to show that 
    \begin{equation}\label{eikonal-value-nakayasu}
    \tilde{u}(x,t):=\inf_{\gamma\in\A^\ast_x}\left\{u_0(\gamma(t))+\int_0^tf(\gamma(\sigma),t-\sigma)\;d\sigma\right\},\qquad\text{for}\ (x,t)\in\X\times[0,T)    
    \end{equation}
    is also a bounded Monge solution of \eqref{eq special} satisfying \eqref{initial-conti0}. In general, we can show that any function $w\in S(\Y_T)$ satisfying, for any $(x, t)\in \Y_T$ and $\delta\in [0, t)$, 
    \[
    \begin{aligned}
    w(x,t)& \leq \inf_{\gamma\in \A^\ast_x}\left\{\int_0^\delta f(\gamma(\sigma),t-\sigma)\,  d\sigma + w(\gamma(\delta),t-\delta)\right\}\\
    &\left(\text{resp.,} \ w(x,t)\geq \inf_{\gamma\in \A^\ast_x}\left\{\int_0^\delta f(\gamma(\sigma),t-\sigma)\,  d\sigma + w(\gamma(\delta),t-\delta)\right\}
    \right)
        \end{aligned}
    \]
    is a Monge subsolution (resp., Monge supersolution) of \eqref{eq special}.
    Then, thanks to the comparison result, Theorem~\ref{thm:comparison1}, we end up with $u=\tilde{u}$ in $\X\times [0, T)$ for $u$ given by \eqref{eq value_f0}.
\end{rmk}



\section{The case with superlinear Hamiltonians}\label{sec:superlinear}

\subsection{Definition of Monge solutions}
Following our study of the eikonal-type equations, we propose a notion of Monge solutions to the Hamilton-Jacobi equation \eqref{eq general}. We assume that $H\in C(\X\times (0, T)\times [0, \infty))$ satisfies (H1) and the associated Lagrangian $L\in C(\X\times (0, T)\times [0, \infty))$ in \eqref{legendre0} fulfills (H2)--(H5). 
In view of (H2) and (H3), we see that $L(x, t, 0)$ is bounded for $(x, t)\in \X\times (0, T)$. Let us  take
\begin{equation}\label{L0}
L_0:=\sup_{x\in \X, t\in (0, T)}|L(x, t, 0)|.
\end{equation}

Suppose that our solution $u\in S(\Y_T)$ for $\Y_T=\X\times (0, T)$.  We can adopt the same procedure as in the case of \eqref{eq special} and make the change of unknown \eqref{unknown-change} with $k\geq 0$ that ensures $u\in S_k(\Y_T)$. We then see that formally $v$ solves
\[
\partial_t v+H(x, t, |\nabla v|)=|\partial_t v|+H(x, t, |\nabla v|)=k \quad \text{in $\X\times (0, T)$.}
\]
Our definition of Monge solutions of \eqref{eq general} is as follows. 

\begin{defi}\label{def monge general}
A function $u\in S(\Y_T)$  is said to be a Monge solution (resp., Monge subsolution, Monge supersolution) of \eqref{eq general} if there exists $k\geq 0$ such that $u\in S_k(\Y_T)$ and $v\in S_0(\Y_T)$ given by \eqref{unknown-change} satisfies 
\begin{equation}\label{subslope-gen}
 |D_L^-v|(x, t)= k\quad \text{(resp., $\leq$, $\geq$)}    
\end{equation}
for any $(x, t)\in \Y_T$, where $|D_L^-v|(x, t)$ at $(x, t)\in \Y_T$ is defined by \eqref{def DL-into}. 
\end{defi}
It is not difficult to see that an equivalent expression of $|D_L^-v|(x, t)$ in \eqref{def DL-into} is 
\begin{equation}\label{def DL-2}
\begin{aligned}
|D_L^- v|(x, t)=\lim_{\delta \to 0+}\sup\bigg\{\frac{v(x, t)-v(y, s)}{t-s} & \ -  L\left(x, t, \frac{d(x, y)}{s}\right):\\
& (y,s)\in \Y_T,\ t-\delta\leq s<t, \ 0<d(x,y)\leq \delta\bigg\}.
\end{aligned}
\end{equation}

\begin{rmk}
Following Lemma \ref{ind_of_k}, one can show that our definition of Monge solutions (resp., Monge subsolutions, Monge supersolutions) does not depend on any particular choice of $k\geq 0$ among those satisfying $u\in S_k(\Y_T)$. 
\end{rmk}

\subsection{Comparison principle}

Next, we study the well-posedness of the initial value problem for \eqref{eq general}\eqref{initial} using the notion of Monge solutions. As a first step, we establish a comparison principle.

\begin{thm}[Comparison principle]\label{thm:comparison-gen}
    Let $(\X,d)$ be a complete length space and $T>0$. 
    Assume that $H\in C(\X\times (0, T)\times [0, \infty))$ satisfies (H1) and $L\in C(\X\times (0, T)\times [0, \infty))$ given by \eqref{legendre0} satisfies (H2).   
    Let $u_1, u_2\in S(\Y_T)$ be respectively a bounded Monge subsolution and a bounded Monge supersolution of \eqref{eq general}. Assume in addition that $u_2\in \lip(\Y_T)$. If \eqref{comp1} holds, 
    then \(u_1\leq u_2\) in $\Y_T$. 
\end{thm}

\begin{proof}
Our argument below is similar to that of Theorem~\ref{thm:comparison1}. We take $k\geq 1$ such that $u_1, u_2\in S_k(\Y_T)$. Let 
$v_1, v_2\in S_0(\Y_T)$ be as in \eqref{unknown-change-i} such that \eqref{comp1-new} holds. 
Since $u_2\in \lip(\Y_T)$, so does $v_2\in \lip(\Y_T)$. Then there exists $C>0$ such that 
    \begin{equation}\label{v2-lip}
    |v_2(x, t)-v_2(y, s)|\leq C(d(x, y)+|t-s|)\quad \text{for all $(x, t), (y, s)\in \Y_T$.}
    \end{equation}
 By the coercivity condition \eqref{L-cocercive} in (H2), we can find $R>1$ such that 
\begin{equation}\label{coercive-R}
\sup_{(x, t)\in \Y_T}C(q+1)-L(x, t, q) \leq 0 \quad \text{for all $q\geq R$.}
\end{equation}
We fix this constant $R>1$ for our later use.

To show that $v_1\leq v_2$ in $\Y_T$, we assume by contradiction that $\sup_{\Y_T} ( v_1-v_2)>\mu$ holds for some \(\mu>0\).  Let $\kappa>0$ and consider again 
    \[
    \Phi(x,t)= v_1(x,t)-v_2(x,t)-\dfrac{\kappa}{T-t},\quad \text{for}\ (x,t)\in \Y_T.
    \]
  For sufficiently small $\kappa>0$, we have $\sup_{\Y_T}\Phi>\mu$. In view of \eqref{comp1-new} and the boundedness of $v_1, v_2$, there exists $0<\eta<T$ sufficiently small such that \eqref{th1}\eqref{th2} hold. Then, there exists $z_0=(x_0,t_0)\in \X\times (2\eta, T-2\eta)$ such that 
    \[
    \Phi(z_0)\geq \sup_{\X\times [\eta, T-\eta]}\Phi- \varepsilon^2>\mu
    \]
    for $\vep>0$ small. Here we can choose $\vep\in (0, \eta)$ small enough such that 
    \begin{equation}\label{eps-small-gen}
    \vep R< \frac{\kappa}{T^2},    
    \end{equation}
    where $R>1$ is the constant satisfying \eqref{coercive-R}. 
    
Note that $(\X\times [\eta, T-\eta],\bar{d})$ is a complete metric space with the  metric $\bar{d}$ given by \eqref{metric spacetime}. By Ekeland's variational principle, we can find $z_\vep=(x_{\varepsilon},t_{\varepsilon})\in \overline{B_{\varepsilon}(x_0,t_0)}\subset \X\times (\eta, T-\eta)$ such that  $z\mapsto \Phi(z)-\varepsilon \bar{d}(z, z_\vep)$ attains a local maximum in $\X\times[\eta, T-\eta]$ at $z=z_\vep$.     It then follows that for any $\delta>0$ small, 
\[
        v_2(z_\vep)-v_2(z)\leq v_1(z_\vep)- v_1(z) +\dfrac{\kappa}{T-t}-\dfrac{\kappa}{T-t_\varepsilon}+\varepsilon \bar{d}(z_\vep, z)
\]
and thus
    \begin{equation}\label{th4-gen}
        \begin{aligned}
   \frac{v_2(z_\vep)-v_2(z)}{t_\vep- t} & -L\left(x_\vep, t_\vep, \frac{d(x_\vep, x)}{t_\vep-t}\right) -\frac{1}{t_\vep-t}\left(\dfrac{\kappa}{T-t}-\dfrac{\kappa}{T-t_\varepsilon}\right)\\
 & \leq  \frac{ v_1(z_\vep)-v_1(z)}{t_\vep-t} -L\left(x_\vep, t_\vep, \frac{d(x_\vep, x)}{t_\vep-t}\right)  + \varepsilon \max\left\{ \frac{d(x_\vep, x)}{t_\vep-t}, 1\right\}
    \end{aligned}
    \end{equation}
    holds for all $z=(x, t)$ such that $d(x, x_\vep)\leq \delta$, $0<t_\vep-t\leq \delta$. 

Applying the definition of Monge supersolution on $u_2$, we see that there exists a sequence $(y_j, s_j)$ with $s_j<t_\vep$ and $\bar{d}((x_\vep, t_\vep), (y_j, s_j))\leq 1/j$ such that, as $j\to \infty$,  
\[
\frac{v_2(x_\vep, t_\vep)-v_2(y_j, s_j)}{t_\vep- s_j} -L\left(x_\vep, t_\vep, \frac{d(x_\vep, y_j)}{t_\vep-s_j}\right)\to |D_L^-v_2|(x_\vep,t_\vep) \geq k.
\]
In view of the Lipschitz continuity of $v_2$ as in \eqref{v2-lip}, we have 
\[
\begin{aligned}
0< \frac{v_2(x_\vep, t_\vep)-v_2(y_j, s_j)}{t_\vep- s_j}\; -&\ L\left(x_\vep, t_\vep, \frac{d(x_\vep, y_j)}{t_\vep-s_j}\right)\\
&\leq C \left(\frac{d(x_\vep, y_j)}{t_\vep-s_j}+1\right)-L\left(x_\vep, t_\vep, \frac{d(x_\vep, y_j)}{t_\vep-s_j}\right) 
\end{aligned}
\]
and by the choice of $R>1$ in \eqref{coercive-R} derived from the coercivity condition \eqref{L-cocercive}, we get $d(x_\vep, y_j)\leq R(t_\vep-s_j)$ for all $j\geq 1$ large.

Taking the limit of \eqref{th4-gen} as $j\to \infty$ along the sequence $(x, t)=(y_j, s_j)$, we obtain 
\[
|D_L^-v_2|(x_\vep, t_\vep)+\frac{\kappa}{T^2}\leq |D_L^-v_2|(x_\vep, t_\vep)+ \frac{\kappa}{(T-t_\vep)^2}\leq |D_L^- v_1|(x_\vep, t_\vep)+\vep R. 
\]
  Applying the definition of Monge subsolutions and Monge supersolutions to $u_1$ and $u_2$ respectively, we are led to $\kappa/T^2\leq \vep R$,
 which is a contradiction to \eqref{eps-small-gen}.
    \end{proof}

\subsection{Existence of Monge solutions}\label{sec:suplin-exist}

We establish our existence result by considering a control-based formula
\begin{equation}\label{value_funct}
u(x,t)=\inf_{\gamma\in \mathcal{A}_x(\X)}\left\{\int_0^t L(\gamma(\sigma),t-\sigma,|\gamma'|(\sigma))d\sigma +u_0(\gamma(t))\right\}    \quad \text{for $x\in \X$, $t\in [0, T)$}
\end{equation}
where $\mathcal{A}_x(\X)$ is given by \eqref{admissible_set_gen}. We show that $u$ given by this formula is a Monge solution of the initial value problem \eqref{eq general}\eqref{initial}. 

It is clear from \eqref{legendre0} that $q\mapsto L(x,t, q)$ is nondecreasing and convex in $[0, \infty)$ for any $(x, t)\in \X\times (0, T)$. In addition, by the coercivity of $H$ in \eqref{H-cocercive}, (H2), and (H3), we have
\begin{equation}\label{L1}
    L_1:=\sup_{(x,t)\in\Y_T}|L(x,t,1)|<+\infty.
\end{equation}
To see this, we first use \eqref{H-cocercive} to get the existence of $R>0$ such that $p-H(x,t,p)<0$ for all $(x,t)\in\Y_T$ and all $p\geq R$. Moreover, by \eqref{legendre0}, $L(x,t,0)\geq -H(x,t,p)$ for all $p\geq 0$. As a consequence, recalling $L_0$ from \eqref{L0}, we have
\[
-L_0\leq L(x,t,1)\leq \max\left\{0,\sup_{0\leq p<R}\{p-H(x,t,p)\}\right\}\leq R+|L(x,t,0)|\leq R+L_0.
\]

We adapt the standard control-based arguments in the Euclidean space to our setting for Monge solutions. Similar solution formulas (with $L$ independent of $t$) are also provided in \cite{Na1} and \cite{GaS}, but our approach is not the same as theirs due to the distinct notions of solutions. 

We first show the boundedness of $u$ given by \eqref{value_funct}. 

\begin{prop}[Boundedness]\label{prop bdd}
Let $(\X,d)$ be a complete length space and $T>0$. Assume that $L\in C(\X\times (0, T)\times [0, \infty))$ given by \eqref{legendre0}
 satisfies (H2)(H3). Assume that $u_0\in C(\X)$ is bounded. Then the function $u$ given by \eqref{value_funct} satisfies
\[
 -L_0 t +\inf u_0 \leq u(x,t)\leq L_0t+u_0(x).
 \]
for all $(x,t)\in \X\times[0,T)$. 
\end{prop}
\begin{proof}
    Fix $(x,t)\in \X\times[0,T)$. By considering a constant curve $\gamma(\sigma)=x$ for all $\sigma\in [0, \infty)$ in the formula \eqref{value_funct}, we obtain 
\[
    u(x,t)\leq \int_0^tL(x,t-\sigma,0)d\sigma+u_0(x) \leq L_0 t+u_0(x),
\]
which gives the desired upper bound. On the other hand, for any $\varepsilon>0$ small, there exists $\gamma_\vep\in \mathcal{A}_x(\X)$ such that
 .  \[
        u(x,t)\geq \int_0^tL(\gamma_\vep(\sigma), t-\sigma,|\gamma_\vep'|(\sigma))\, d\sigma+u_0(\gamma_\vep(t))-\varepsilon
    \geq -L_0 t+\inf u_0-\varepsilon.
    \]
    Letting $\varepsilon\to 0$, we obtain the lower bound.
\end{proof}

One can obtain the dynamic programming principle as below. 

\begin{lem}[Dynamic programming principle]\label{lem dpp}
 Let $(\X,d)$ be a complete length space and $T>0$. Assume that $L\in C(\X\times (0, T)\times [0, \infty))$  satisfies (H2)(H3). Let $u_0\in C(\X)$ and $u$ be given by \eqref{value_funct}. Then,  $u$ satisfies the following dynamic programming principle:
\begin{equation}\label{dpp}
u(x,t)=\inf_{\gamma\in\mathcal{A}_x(\X)}\left\{\int_0^\delta L(\gamma(\sigma), t-\sigma, |\gamma'|(\sigma))\,  d\sigma +u(\gamma(\delta),t-\delta)\right\}\quad \text{for all}\ \delta\in[0,t].
\end{equation}
\end{lem}

We omit the proof, since \eqref{dpp} is a straightforward generalization of \eqref{c-sol dpp} studied in \cite{Na1} and can be easily derived from \eqref{value_funct} following the proof of Lemma \ref{lem dpp1}. The following proposition states the time Lipschitz regularity of $u$ in \eqref{value_funct}. 

\begin{prop}[Lipschitz continuity in time]\label{prop gen-lip-time}
 Let $(\X,d)$ be a complete length space and $T>0$. Assume that $L\in C(\X\times (0, T)\times [0, \infty))$ 
 satisfies (H2)(H3)(H5). 
 Let $u_0\in \lip(\X)$ and $u$ be given by \eqref{value_funct}.  Then,  there exists $K>0$ such that 
\begin{equation}\label{gen-time-lip}
|u(x,t)-u(x,s)|\leq K|t-s| \quad \text{for all $x\in \X$ and $t, s\in [0, T)$}.
\end{equation} 
In particular, \eqref{initial-conti} holds for such $K>0$. 
\end{prop}

\begin{proof}
    Let $x\in \X$ and $t_1,t_2\in[0,T)$ with $t_1<t_2$. Recall the constant $L_0$ in \eqref{L0}. Using in \eqref{dpp} the constant curve $\gamma(\sigma)=x$ for all $\sigma \in [0,\infty)$, we have 
    \begin{equation}\label{up-bound-u}
        u(x,t_2)-u(x,t_1)\leq \int_0^{t_2-t_1}L(\gamma(\sigma),t_2-\sigma,0)\, d\sigma \leq L_0 |t_2-t_1|.
    \end{equation}
 For any $0<\varepsilon<1$ small, there exists $\gamma_{\varepsilon}\in\mathcal{A}_x(\X)$ such that 
 \begin{equation}\label{monge-lip eq1}
 \begin{aligned}
 & u(x,t_2)\geq \int_0^{t_2}L(\gamma_\vep(\sigma), t_2-\sigma, |\gamma_\vep'|(\sigma)) d\sigma +u_0(\gamma_{\varepsilon}(t_2))-\varepsilon|t_2-t_1|\\    
 & \geq \int_0^{t_1} L(\gamma_\vep(\sigma), t_2-\sigma, |\gamma_\vep'|(\sigma)) d\sigma +\int_{t_1}^{t_2} L(\gamma_\vep(\sigma), t_2-\sigma, |\gamma_\vep'|(\sigma)) d\sigma+u_0(\gamma_{\varepsilon}(t_2))-\varepsilon|t_2-t_1|.
 \end{aligned}
 \end{equation}
By (H5) and the Lipschitz continuity of $u_0$, we get
\[
\begin{aligned}
&\int_0^{t_1} L(\gamma_\vep(\sigma), t_2-\sigma, |\gamma_\vep'|(\sigma))\, d\sigma +u_0(\gamma_\vep(t_2))\\
&\geq \int_0^{t_1} L(\gamma_\vep(\sigma), t_1-\sigma, |\gamma_\vep'|(\sigma))\, d\sigma -t_1 C_T|t_2-t_1|+u_0(\gamma_\vep(t_1)) -C_0\ell_\vep,    
\end{aligned}
\]
where $\ell_\vep=\int_{t_1}^{t_2}|\gamma_\vep'|(\sigma)\, d\sigma$,  $C_0\geq 0$ is the Lipschitz constant of $u_0$ and $C_T>0$ is the constant appearing in \eqref{L-lip2}. 
It follows from the definition of $u(x, t_1)$ that
\begin{equation}\label{monge-lip eq2}
\int_0^{t_1} L(\gamma_\vep(\sigma), t_2-\sigma, |\gamma_\vep'|(\sigma))\, d\sigma +u_0(\gamma_\vep(t_2))\geq u(x, t_1)-t_1 C_T|t_2-t_1|-C_0\ell_\vep.
\end{equation}
In addition, applying (H2) and Jensen's inequality we have 
\begin{equation}\label{monge-lip eq3}
\int_{t_1}^{t_2} L(\gamma_\vep(\sigma), t_2-\sigma, |\gamma_\vep'|(\sigma)) d\sigma\geq \int_{t_1}^{t_2} m(|\gamma_\vep'|(\sigma))\, d\sigma\geq (t_2-t_1) m\left(\frac{\ell_\vep}{t_2-t_1}\right). 
\end{equation}
Combining \eqref{monge-lip eq1}, \eqref{monge-lip eq2}, and \eqref{monge-lip eq3}, we get 
\begin{equation}\label{monge-lip eq4}
 u(x, t_2)\geq u(x,t_1) +(t_2-t_1) m\left(\frac{\ell_\vep}{t_2-t_1}\right)-(TC_T+\vep)|t_2-t_1|-C_0\ell_\vep. 
\end{equation}
Noticing $\eqref{up-bound-u}$, we thus have  
\[
m\left(\frac{\ell_\vep}{t_2-t_1}\right)\leq L_0 +TC_T+1+C_0 \frac{\ell_\vep}{t_2-t_1},
\]
which by the coercivity of $m$ in \eqref{L-cocercive} yields $\ell_\vep\leq R|t_2-t_1|$ for some $R>0$. 
In view of \eqref{monge-lip eq4} again, taking $\vep\to 0$, we are led to 
\[
u(x, t_2)\geq u(x,t_1)-\tilde{C}|t_2-t_1|,
\]
where 
\[
\tilde{C} =C_0 R +TC_T+\sup_{q\in [0, R]} |m(q)|. 
\]
We conclude the proof of \eqref{gen-time-lip} by taking $K=\tilde{C}+L_0$. 
\end{proof}

We next prove that $u$ defined by \eqref{value_funct} in Lipschitz continuous in $\X\times [0, T)$.   
\begin{prop}[Lipschitz regularity]\label{prop gen-lip}
Let $(\X,d)$ be a complete length space and $T>0$. Assume that $H\in C(\X\times (0, T)\times [0, \infty))$ satisfies (H1) and $L\in C(\X\times (0, T)\times [0, \infty))$ satisfies (H2)(H3)(H5). Let $u_0\in\lip(\X)$.  Then the function $u$ given by \eqref{value_funct} is Lipschitz continuous in $\X\times[0,T)$.
\end{prop}
\begin{proof}
    We have shown in Proposition \ref{prop gen-lip-time} that there exists $K>0$ such that 
\begin{equation}\label{gen-lip eq2}
|u(x,t)-u(x,s)|\leq K|t-s| \quad \text{for all $x\in \X$ and $t, s\in [0, T)$}.
\end{equation}
It remains to show $u(\cdot,t)\in \lip(\X)$ with a Lipschitz constant uniformly for all $t\in (0,T)$.

Fix arbitrarily $x,y\in \X,\ t\in(0,T)$ such that 
\begin{equation}\label{gen-lip eq3}
0<d(x,y)<T-t.    
\end{equation}
For any $\delta>0$ such that $\delta<(T-t)-d(x,y)$, there exists $\gamma_1\in \mathcal{A}_x(\X)$ such that 
\begin{equation}\label{gen-lip eq1}
u(x,t)\geq \int_0^tL(\gamma_1(\sigma),t-\sigma,|\gamma_1'|(\sigma))d\sigma + u_0(\gamma_1(t))-\delta.
\end{equation}
There also exists a curve $\xi:[0,\tau]\to\X$ with $\tau>0$ such that $\xi(0)=y$, $\xi(\tau)=x$, $\tau \leq d(x, y)+\delta$, and $|\xi'|=1$ a.e in $(0,\tau)$. Next, we build a curve $\gamma_2\in\mathcal{A}_y(\X)$ as below:
\[
\gamma_2(\sigma)=\begin{dcases}
\xi(\sigma) & \text{if $0\leq \sigma\leq \tau$,}\\
\gamma_1(\sigma-\tau) & \text{if $\tau< \sigma\leq t+\tau$,}\\
\gamma_1(t) & \text{if $t+\tau<\sigma$.}
\end{dcases}
\]
Then by the definition of $u$ we have 
\[
    \begin{aligned}
       &  u(y,t+\tau)\leq \int_0^{t+\tau}L(\gamma_2(\sigma),t+\tau-\sigma,|\gamma_2'|(\sigma))\, d\sigma+u_0(\gamma_2(t+\tau))\\
    =& \int_0^{\tau}L(\gamma_2(\sigma),t+\tau-\sigma, |\gamma_2'|(\sigma))d\sigma+\int_{\tau}^{t+\tau}L(\gamma_2(\sigma),t+\tau-\sigma,|\gamma_2'|(\sigma))\, d\sigma+u_0(\gamma_2(t+\tau))\\
    =& \int_0^{\tau}L(\xi(\sigma),t+\tau-\sigma,1)\, d\sigma+\int_{0}^{t}L(\gamma_1(\sigma),t-\sigma,|\gamma_1'|(\sigma))\, d\sigma+u_0(\gamma_1(t)).
    \end{aligned}
\]
It follows from \eqref{L1} and \eqref{gen-lip eq1} that 
\[
u(y,t+\tau)-u(x,t)\leq L_1\tau+\delta.
\]
Using \eqref{gen-lip eq2}, we further get
\[
\begin{aligned}
   u(y,t)-u(x,t)& =u(y,t)-u(y,t+\tau)+u(y,t+\tau)-u(x,t)\\
   &\leq (K+L_1)\tau+\delta \leq (K+L_1)(d(x,y)+\delta)+\delta.
\end{aligned}
\]
Letting $\delta\to 0$ and exchanging the roles of $x$ and $y$, we obtain
\[
|u(x,t)-u(y,t)|\leq K_1d(x,y)
\]
for all $x, y\in \X$, $t\in (0, T)$ satisfying \eqref{gen-lip eq3}, where $K_1=K+L_1$. 

Finally, let us prove the inequality for a general pair $x, y\in \X$ without the constraint \eqref{gen-lip eq3}.
Since $\X$ is a length space, for any $x,y\in \X,\ t\in(0,T)$ and $\varepsilon>0$ small, there exists an arc-length parametrized curve $\gamma:[0, s]\to\X$ such that $s\leq d(x,y)+\varepsilon$, $\gamma(0)=x$, and $\gamma(s)=y$. Let $\{s_0, s_1,\ldots, s_n\}$ be a partition of $[0,s]$ such that $s_0=0$, $s_n=s$, and $|s_i-s_{i-1}|<T-t$ for all $i=1, 2, \ldots, n$. Then, applying our result above, we get
\[
\begin{aligned}
    |u(x,t)-u(y,t)|& \leq \sum_{i=1}^{n}|u(\gamma(s_i),t)-u(\gamma(s_{i-1}),t)|\\
    &\leq \sum_{i=1}^n K_1d(\gamma(s_{i}),\gamma(s_{i-1}))\leq K_1 s\leq K_1(d(x,y)+\varepsilon).
\end{aligned}
\]
We conclude our proof by letting $\varepsilon\to 0$.
 \end{proof}

 We finally prove that $u$ given by \eqref{value_funct} is a Monge solution of the initial value problem \eqref{eq general}\eqref{initial}. 
\begin{thm}[Existence of Monge solutions]\label{thm:gen-Monge}
     Let $(\X,d)$ be a complete length space and $T>0$. Assume that $H\in C(\X\times (0, T)\times [0, \infty))$ satisfies (H1) and $L\in C(\X\times (0, T)\times [0, \infty))$ satisfies (H2)--(H5). Let $u_0\in\lip(\X)$ be bounded.  Let $u$ be given by \eqref{value_funct}. Then, $u\in \lip(\X\times [0, T))$ is a bounded Monge solution of \eqref{eq general} in the sense of Definition~\ref{def monge general}. 
\end{thm}
We have already shown in Proposition \ref{prop bdd}, Proposition \ref{prop gen-lip-time}  and Proposition \ref{prop gen-lip} that $u$ given by \eqref{value_funct} is bounded and Lipschitz in $\X\times [0, T)$. 
It thus suffices to verify the properties of Monge subsolutions and Monge supersolutions for $u$, which are discussed in the following two propositions. 

\begin{prop}[Monge subsolution property]\label{prop gen-monge1}
    Let $(\X,d)$ be a complete length space and $T>0$. Assume that $L\in C(\X\times (0, T)\times [0, \infty))$ satisfies (H2)(H4). Suppose that $u\in \lip(\Y_T)$ satisfies 
    \begin{equation}\label{control-sub}
    u(x, t)\leq \inf_{\gamma\in \A_x(\X)}\left\{\int_0^\delta L(\gamma(\sigma),t-\sigma,|\gamma'|(\sigma))d\sigma + u(\gamma(\delta),t-\delta)\right\}
    \end{equation}
    for any $(x, t)\in \Y_T$ and $0\leq \delta<t$.
    Then $u$ is a Monge subsolution of \eqref{eq general} in the sense of Definition~\ref{def monge general}. In particular, if (H2)--(H5) hold and $u_0\in \lip(\X)$ is bounded, $u$ given by \eqref{value_funct} is a bounded Monge subsolution of \eqref{eq general}.
\end{prop}
\begin{proof}
Suppose that $u$ is $C$-Lipschitz in $\Y_T$ with $C>0$.  Choose $k\geq C$. 
It is clear that $u\in S_k(\Y_T)$. Let $v$ be given by \eqref{unknown-change}. 

Fix $z=(x,t)\in\Y_T=\X\times (0, T)$ arbitrarily. In view of \eqref{control-sub}, for any $\gamma\in\mathcal{A}_x(\X)$
    \begin{equation}\label{c2m eq5}
    u(z)\leq \int_0^\delta L(\gamma(\sigma),t-\sigma,|\gamma'|(\sigma))d\sigma + u(\gamma(\delta),t-\delta)
    \end{equation}
    holds for all $\delta\in(0,t)$. Let $(y_\delta, s_\delta)\in\Y_T$ with $0<t-s_\delta \leq \delta$ and $d(x, y_\delta)\leq \delta$. We can find a curve $\ol{\gamma}_\delta\in \mathcal{A}_x(\X)$ with a constant speed $c_\delta\geq 0$ a.e. in $(0, t-s_\delta)$ such that $\ol{\gamma}_\delta(0)=x,\;\ol{\gamma}_\delta(t-s_\delta)=y_\delta$, and 
    \[
     (t-s_\delta) c_\delta  =\ell(\ol{\gamma}_\delta|_{[0, t-s_\delta]})\leq (1+\delta)d(x,y_\delta). 
    \]
     By the Lipschitz continuity of $u$, we get
    \[
    \frac{v(x, t)-v(y_\delta, s_\delta)}{t-s_\delta}\leq \dfrac{u(x,t)-u(y_\delta, s_\delta)}{t-s_\delta}+k \leq C+k+C\frac{d(x, y_\delta)}{t-s_\delta}. 
    \]
We use the expression \eqref{def DL-2} to estimate $ |D^-_Lv|(x, t)$ by dividing our discussions into two cases regarding the limit of $d(x, y_\delta)/(t-s_\delta)$. 

Case (i): Suppose that $d(x, y_\delta)/(t-s_\delta)\to \infty$ as $\delta\to 0+$. Then by (H2) we obtain 
    \begin{equation}\label{c2m add1}
    \begin{aligned}
    &\limsup_{\delta\to 0+}\left\{\frac{v(x, t)-v(y_\delta, s_\delta)}{t-s_\delta}-L\left(x,t,\dfrac{d(x,y_\delta)}{t-s_\delta}\right)\right\}\\
    &\leq \limsup_{\delta\to 0+}\left\{(C+k)\left(1+\frac{d(x, y_\delta)}{t-s_\delta}\right)- L\left(x,t,\dfrac{d(x,y_\delta)}{t-s_\delta}\right) \right\} \leq 0<k.   
    \end{aligned}
    \end{equation}
Case (ii): Suppose that $\limsup_{\delta\to 0+}d(x, y_\delta)/(t-s_\delta)\leq R$
    for some $R>0$. In this case, we have $c_\delta\leq 2R$ for all $\delta>0$ small. 
    By \eqref{c2m eq5} and the monotonicity of $q\mapsto L(x, t, q)$, we get 
    \[
    \begin{aligned}
         \dfrac{v(x, t)-v(y_\delta, s_\delta)}{t-s_\delta}-L\left(x,t,\dfrac{d(x,y_\delta)}{t-s_\delta}\right)&\leq \dfrac{u(x,t)-u(\ol{\gamma}_\delta(t-s_\delta), s_\delta)}{t-s_\delta}+k-L\left(x,t, \frac{c_\delta}{1+\delta}\right)\\
         \leq \ &  k+\dfrac{1}{t-s_\delta}\int_0^{t-s_\delta}L(\ol{\gamma}_\delta(\sigma), t-\sigma, c_\delta)\, d\sigma-L\left(x,t, \frac{c_\delta}{1+\delta}\right). 
        \end{aligned}
        \]
Since $(\ol{\gamma}_\delta(\sigma), t-\sigma, c_\delta)$ stays in a bounded subset of $\X\times (0, T)\times [0, \infty)$ for all $\sigma\in [0, t-s_\delta]$ with $\delta>0$ small, by the local uniform continuity of $L$ in (H4), there exists a modulus of continuity $\omega_R$ such that 
\[
\left|L(\ol{\gamma}_\delta(\sigma), t-\sigma, c_\delta)-L\left(x,t, \frac{c_\delta}{1+\delta}\right)\right|\leq \omega_R\left((1+\delta)d(x,y_\delta)+\delta+\frac{\delta}{1+\delta}c_\delta\right). 
\]
It then follows that
        \[
         \dfrac{v(x, t)-v(y_\delta, s_\delta)}{t-s_\delta}-L\left(x,t,\dfrac{d(x,y_\delta)}{t-s_\delta}\right)\leq k+\omega_R\left((1+\delta)d(x,y_\delta)+\delta+\frac{2\delta R}{1+\delta}\right)
    \]
    for all $\delta>0$ small. Sending $\delta\to 0$ and combining it with \eqref{c2m add1} for Case (i), we obtain $|D^-_Lv|(z)\leq k$. 
    Hence, $u$ is a Monge subsolution of \eqref{eq general}.

For the function $u$ defined by \eqref{value_funct} with $u_0\in \lip(\X)$ bounded, by Lemma \ref{lem dpp} we see that $u$ satisfies \eqref{control-sub} and therefore is a Monge subsolution of \eqref{eq general}. Its boundedness is established in Proposition \ref{prop bdd}. 
    \end{proof}

 \begin{prop}[Monge supersolution property]\label{prop gen-monge2}
    Let $(\X,d)$ be a complete length space and $T>0$. Assume that $L\in C(\X\times (0, T)\times [0, \infty))$ satisfies (H2)(H4).  If $u\in \lip(\Y_T)$ and satisfies 
    \begin{equation}\label{control-super}
    u(x, t)\geq \inf_{\gamma\in \A_x(\X)}\left\{\int_0^\delta L(\gamma(\sigma),t-\sigma,|\gamma'|(\sigma))d\sigma + u(\gamma(\delta),t-\delta)\right\}
    \end{equation}
    for any $(x, t)\in \Y_T$ and $0\leq \delta<t$. 
    Then $u$ is a Monge supersolution of \eqref{eq general} in the sense of Definition~\ref{def monge general}. In particular, if (H2)--(H5) hold and $u_0\in \lip(\X)$ is bounded, $u$ given by \eqref{value_funct} is a bounded Monge supersolution of \eqref{eq general}.
\end{prop}
\begin{proof}
    Assuming that $u$ is $C$-Lipschitz in $\Y_T$ for $C>0$, we again take $v$ as in \eqref{unknown-change} with $k\geq C$ such that $u\in S_k(\Y_T)$ and $v\in S_0(\Y_T)$. 

    Fix $z=(x, t)\in \Y_T$ arbitrarily. By \eqref{control-super}, for any $\delta\in (0, t)$ small, there exists $\gamma_\delta\in\mathcal{A}_x(\X)$ such that 
    \begin{equation}\label{c2m eq1}
    u(z)\geq \int_0^{\delta}L(\gamma_\delta(\sigma),t-\sigma,|\gamma_\delta'|(\sigma))d\sigma + u(\gamma_\delta(\delta),t-\delta)-\delta^2.     
    \end{equation}
 We take 
 \[
 \ell_{\delta}=\ell(\gamma_\delta|_{[0, \delta]})=\int_0^{\delta}|\gamma_\delta'|(\sigma) d\sigma.
 \]
 It is clear that $\ell_\delta\geq d(x, \gamma_\delta(\delta))$. By (H2), there exists a convex coercive function $m\in C(\R)$ such that 
\[
\frac{1}{\delta}\int_0^{\delta}L(\gamma_\delta(\sigma),t-\sigma,|\gamma_\delta'|(\sigma))d\sigma \geq \frac{1}{\delta}\int_0^{\delta} m(|\gamma_\delta'|(\sigma))d\sigma.
\]
By \eqref{c2m eq1} and the Lipschitz continuity of $u$, we obtain 
\begin{equation}\label{c2m eq-add1}
    \frac{1}{\delta}\int_0^{\delta} m(|\gamma_\delta'|(\sigma))d\sigma\leq \dfrac{u(x,t)-u(\gamma_\delta(\delta),t-\delta)}{\delta}+\delta\leq C\left(1+\dfrac{\ell_\delta}{\delta}\right)+\delta. 
\end{equation}
 It follows from Jensen's inequality that 
\[
 m\left(\frac{\ell_\delta}{\delta} \right)=m\left(\frac{1}{\delta} \int_0^{\delta}|\gamma_\delta'|(\sigma)\, d\sigma\right) \leq \frac{1}{\delta}\int_0^{\delta} m(|\gamma_\delta'|(\sigma))d\sigma\leq C\left(1+\dfrac{\ell_\delta}{\delta}\right)+\delta, 
 \]
 which by the coercivity of $m$ in \eqref{L-cocercive}, there exists $R>1$ such that $\ell_\delta\leq R\delta$ for all $0<\delta<1$ small. This yields $\ell_\delta\to 0$ as $\delta\to 0$. 

Noting that $q\mapsto L(x, t, q)$ is nondecreasing for all $(x, t)\in \Y_T$, by \eqref{c2m eq1} we have
\begin{equation}\label{c2m eq6}
\begin{aligned}
        &\dfrac{v(z)-v(\gamma_\delta(\delta),t-\delta)}{\delta}-L\left(x,t,\dfrac{d(x,\gamma_\delta(\delta))}{\delta}\right)\\
        &\geq k-\delta+\frac{1}{\delta}\int_0^{\delta}L\left(\gamma_\delta(\sigma),t-\sigma, |\gamma_\delta'|(\sigma)\right)d\sigma -L\left(x,t, \frac{\ell_\delta}{\delta}\right). 
        \end{aligned}
\end{equation}
Since \eqref{L-lip1} in (H4) yields
\[
L\left(\gamma_\delta(\sigma),t-\sigma, |\gamma_\delta'|(\sigma)\right)\geq L(x, t, |\gamma_\delta'|(\sigma))-\omega_L(\ell_\delta+\delta)\left(1+|m(|\gamma_\delta'|(\sigma))|\right)
\]
for $\sigma\in(0,\delta)$, we then have 
\[
\begin{aligned}
&\frac{1}{\delta}\int_0^{\delta}L\left(\gamma_\delta(\sigma),t-\sigma, |\gamma_\delta'|(\sigma)\right)\, d\sigma\\
&\geq \frac{1}{\delta}\int_0^{\delta}L(x, t, |\gamma_\delta'|(\sigma))\, d\sigma -\omega_L(\ell_\delta+\delta)\left(\frac{1}{\delta}\int_0^{\delta}|m(|\gamma_\delta'|(\sigma))|\, d\sigma+1\right)\\
\end{aligned}
\]
Let $m_0=\inf_{[0,\infty)} m$. 
Note that $|m(q)|\leq m(q)+2|m_0|$ for all $q\geq 0$. Applying \eqref{c2m eq-add1} and Jensen's inequality to the convex function $q\mapsto L(x, t, q)$, we obtain 
\[
\frac{1}{\delta}\int_0^{\delta}L\left(\gamma_\delta(\sigma),t-\sigma, |\gamma_\delta'|(\sigma)\right)d\sigma \geq L\left(x,t, \frac{\ell_\delta}{\delta}\right) -2\omega_L(\ell_\delta+\delta)\left(CR+|m_0|+1\right). 
\] 
By \eqref{c2m eq6}, we thus deduce that
\[
\dfrac{v(z)-v(\gamma_\delta(\delta),t-\delta)}{\delta}-L\left(x,t,\dfrac{d(x,\gamma_\delta(\delta))}{\delta}\right)\geq k-\delta-2\omega_L(\ell_\delta+\delta)(CR+|m_0|+1). 
\]
Letting $\delta\to 0$ yields $|D^-_Lv|(z)\geq k$, which verifies the Monge supersolution property of $u$. 
\end{proof}

Since \eqref{initial-conti} holds for the function $u$ defined by \eqref{value_funct}, the well-posedness result in Theorem~\ref{thm:intro2} follows from Theorem~\ref{thm:comparison-gen} and Theorem~\ref{thm:gen-Monge}.



\section{Equivalence with curve-based and slope-based solutions}\label{sec:equivalence}

In this last section, we discuss the relation between curve-based solutions, slope-based solutions, and Monge solutions. The equivalence of these notions is expected at least in the case of typical Hamiltonians. We will rigorously prove the equivalence for \eqref{eq time-ind}, where the Hamiltonian $H$ does not depend on $t$.

\subsection{Equivalence with curve-based viscosity solutions}

Consider \eqref{eq time-ind} with initial value \eqref{initial}. Assume that $H\in C(\X\times [0, \infty))$ satisfies \eqref{H-cocercive}. Under the assumptions (H2)--(H4) for $L$ given by \eqref{legendre0} and bounded Lipschitz initial value, we have shown in Theorem~\ref{thm:intro2} that there exists a unique Monge solution to the initial value problem. (We do not need to explicitly impose (H5), since $H, L$ are independent of $t$.)

Note that the unique Monge solution $u$ is represented by the formula \eqref{value_funct} with the dependence of $L$ on $t$ dropped. It coincides with the representation formula \eqref{value time-ind} of the unique curve-based solution of \eqref{eq time-ind} and \eqref{initial}; see Proposition \ref{curve_subsol} and Proposition \ref{curve_supersol}. We thus can easily establish the equivalence between the notions of Monge solutions and curve-based viscosity solutions in this case. 

\begin{thm}[Equivalence of Monge and curve-based solutions for superlinear Hamiltonians]\label{thm:equiv-cm}
Let $(\X,d)$ be a complete length space and $T>0$. Assume that $H\in C(\X\times [0, \infty))$ satisfies (H1) and $L\in C(\X\times [0, \infty))$ defined by \eqref{legendre0} satisfies (H2)--(H4) (without dependence on $t$). Let $u_0\in\lip(\X)$ be bounded and $u\in C(\X\times [0, T))$ satisfy \eqref{initial} in the sense of \eqref{initial-conti0}. Then, $u$ is a Monge solution of \eqref{eq time-ind} if and only if it is a curve-based viscosity solution of \eqref{eq time-ind}. 
\end{thm}

We actually obtain the following consequences regarding subsolution and supersolution properties under the same assumptions as in Theorem~\ref{thm:equiv-cm}:
\begin{itemize}
    \item Any curve-based subsolution $u\in \lip(\X\times (0, T))$ of \eqref{eq time-ind} is a Monge subsolution, as shown in Proposition \ref{curve_subsol} and Proposition \ref{prop gen-monge1}. 
    \item Any curve-based supersolution $u\in \lip(\X\times (0, T))$ of \eqref{eq time-ind} is a Monge supersolution, as shown in Proposition \ref{curve_supersol} and Proposition \ref{prop gen-monge2}. 
\end{itemize}
It is not clear to us whether the reverse implication holds. 
In the case of \eqref{eq special0}, we can obtain an equivalence result similar to Theorem~\ref{thm:equiv-cm}. 

\begin{thm}[Equivalence of Monge and curve-based solutions of eikonal-type equation]\label{thm:equiv-eikonal1}
    Let $(\X,d)$ be a complete length space and $T>0$. Assume that $f: \X\to \R$ is bounded and continuous. Let $u_0\in\lip(\X)$ be bounded and $u\in C(\X\times [0, T))$ satisfy \eqref{initial} in the sense of \eqref{initial-conti0}. Then, $u$ is a Monge solution of \eqref{eq special0} if and only if it is a curve-based viscosity solution of \eqref{eq special0}. 
\end{thm}
The proof is based on Remark \ref{special_cs_Monge}. In fact, besides \eqref{eq value_f0}, \eqref{eikonal-value-nakayasu} also provides a representation formula for the unique Monge solution of the initial value problem. By \eqref{eikona-hl}, we see that the curve class $\A^\ast_x$ in \eqref{eikonal-value-nakayasu} is consistent with $\A_x$ used in the characterization of curve-based viscosity solutions in the case of \eqref{eq special0}. Hence, for the initial value problem for \eqref{eq special0}, the curve-based viscosity solution is equivalent to the Monge solution.

\subsection{Equivalence with slope-based viscosity solutions}

Let us now proceed to discuss the relation between slope-based and Monge solutions of \eqref{eq time-ind}. Note that by (H1), 
\begin{equation}\label{HL0}
    H(x, t, 0)=-L(x, t, 0) \quad \text{for all $x\in \X$, $t\in (0, T)$.}
\end{equation}

\begin{prop}[Slope-based solution property implied by Monge solutions]\label{semi-eqiuv_1}
    Let $(\X,d)$ be a complete length space and $T>0$. Assume that $H\in C(\X\times [0, \infty))$ satisfies (H1). Let $L\in C(\X\times [0, \infty))$ be defined as in \eqref{legendre0}. If $u\in S(\Y_T)$ is a Monge solution (resp., subsolution, supersolution) of \eqref{eq time-ind} in the sense of Definition \ref{def monge general}, then it is a slope-based viscosity solution (resp., subsolution, supersolution) of \eqref{eq time-ind} in the sense of Definition \ref{def metric special}.
\end{prop}

\begin{proof}
    (Supersolution) Suppose that $u\in S(\Y_T)$ be a Monge supersolution of \eqref{eq time-ind}. Fix $z_0=(x_0,t_0)\in \Y_T$ arbitrarily. There exists $k\geq0$ such that $u\in S_k(\Y_T)$ and $v\in  S_0(\Y_T)$ given by \eqref{unknown-change} satisfies 
    \begin{equation}\label{equiv-sm eq1}
    |D_L^-v|(z_0)\geq k. 
    \end{equation}
    For any $\psi=\psi_1+\psi_2\in \overline{\mathcal{C}}$ as in Definition \ref{def s-sol} such that $u-\psi$ attains a local minimum at $z_0$, then
    \[
    u(z)-\psi_1(z)-\psi_2(z)\geq u(z_0)-\psi_1(z_0)-\psi_2(z_0)
    \]
    for all $z=(x, t)\in\Y_T$ close to $z_0$. Therefore, for $x\neq x_0$ and $t<t_0$,
    \[
    \psi_1(z_0)-\psi_1(z)+\psi_2(z_0)-\psi_2(z)\geq u(z_0)-u(z)=v(z_0)-v(z)-k(t_0-t),
    \]
    which implies
    \[
    \dfrac{\psi(x,t_0)-\psi(x, t)}{t_0-t}+\dfrac{d(x_0, x)}{t_0-t}\left[\sum_{i=1}^2\dfrac{|\psi_i(x_0,t_0)-\psi_i(x,t_0)|}{d(x_0,x)}\right]\geq \dfrac{v(z_0)-v(z)}{t_0-t}-k.
    \]
    In view of \eqref{legendre0}, we get
    \[
    \begin{aligned}
    \dfrac{\psi(x,t_0)-\psi(x, t)}{t_0-t}+ & H\left(x_0, t_0,\left[\sum_{i=1}^2\dfrac{|\psi_i(x_0,t_0)-\psi_i(x, t_0)|}{d(x_0, x)}\right]\right)\\
    &\geq\ \dfrac{v(z_0)-v(z)}{t_0-t}-L\left(x_0,t_0,\dfrac{d(x_0, x)}{t_0-t}\right)-k.
    \end{aligned}
    \]
    Letting $d(x_0, x)\to 0, t\to t_0-$ on both sides, we obtain
    \[
    \psi_t(z_0)+H(x_0, t_0,|\nabla\psi_1|(z_0)+|\nabla\psi_2|(z_0))\geq |D_L^-v|(z_0)-k.
    \]
    Applying \eqref{equiv-sm eq1}, we end up with 
    \[
    \psi_t(z_0)+H^{|\nabla\psi_2|^*(z_0)}(x_0,t_0,|\nabla\psi_1|(z_0))\geq 0,
    \]
    which, by the arbitrariness of $z_0\in \Y_T$, implies that $u$ is a slope-based supersolution of \eqref{eq time-ind}. 
    
    \noindent (Subsolution) Suppose now that $u\in S(\Y_T)$ be a Monge subsolution of \eqref{eq time-ind}. Let $z_0=(x_0,t_0)\in \Y_T$. There exists $k\geq0$ such that $u\in S_k(\Y_T)$ and $v\in S_0(\Y_T)$ given by \eqref{unknown-change} satisfies $|D_L^-v|(z_0)\leq k$. 
    For any $\psi=\psi_1+\psi_2\in \underline{\mathcal{C}}$ as in Definition \ref{def s-sol} such that $u-\psi$ attains a local maximum at $z_0$, then
    \[
    u(z)-\psi_1(z)-\psi_2(z)\leq u(z_0)-\psi_1(z_0)-\psi_2(z_0)
    \]
    for all $z=(x, t)\in\Y_T$ close to $z_0$. In terms of $v$, we obtain
    \begin{equation}\label{equiv-sm eq2}
    \dfrac{\psi(x_0, t_0)-\psi(x, t)}{t_0-t}\leq \dfrac{v(x_0, t_0)-v(x, t)}{t_0-t}-k.
    \end{equation}
    for all $z=(x, t)\in\Y_T$ near $z_0$ with $t<t_0$. 
    Since $p\mapsto H(x, t, p)$ is nondecreasing, we can take 
    \[
    p^*=\max\{p\geq 0 : H(x_0,t_0,p)=H(x_0,t_0,0)\}.
    \]    
    If $|\nabla\psi_1|(z_0)-|\nabla\psi_2|(z_0)\leq p^*$, then by \eqref{H-extension} and \eqref{HL0} we have 
    \[
    H_{|\nabla\psi_2|^*(z_0)}(x_0, t_0,|\nabla\psi_1|(z_0))=H(x_0, t_0, 0)= -L(x_0, t_0, 0).
    \]
    Then, for any $t<t_0$ close to $t_0$ and $x=x_0$, we get from \eqref{equiv-sm eq2}
\[
 \dfrac{\psi(x_0, t_0)-\psi(x_0, t)}{t_0-t}+H_{|\nabla\psi_2|^*(z_0)}\left(x_0,t_0,|\nabla\psi_1|(z_0)\right)   \leq \dfrac{v(x_0, t_0)-v(x_0, t)}{t_0-t}-L(x_0,t_0,0)-k. 
\]
    Letting $t\to t_0$, we obtain
    \begin{equation}\label{equiv-sm eq5}
    \psi_t(x_0,t_0)+H_{|\nabla\psi_2|^*(x_0,t_0)}\left(x_0,t_0,|\nabla\psi_1|(x_0,t_0)\right)\leq |D^-_Lv|(z_0)-k\leq 0. 
    \end{equation}
    
    Let us discuss the case when 
    \begin{equation}\label{equiv-sm eq3}
    |\nabla\psi_1|(z_0)-|\nabla\psi_2|(z_0)> p^*.
    \end{equation} 
    Since $|\nabla^-\psi_1|(z_0)=|\nabla\psi_1|(z_0)$, there exists a sequence $x_j\neq x$ such that
    \[ 
    \dfrac{\psi_1(x_0, t_0)-\psi_1(x_j, t_0)}{d(x_0, x_j)}\to |\nabla\psi_1|(x_0, t_0)\quad\text{as $j\to \infty$. }
    \]
    Let 
    \[
    p_j=\dfrac{\psi_1(x_0, t_0)-\psi_1(x_j, t_0)}{d(x_0, x_j)}-\dfrac{|\psi_2(x_0, t_0)-\psi_2(x_j, t_0)|}{d(x_0, x_j)}.
    \]
    Suppose that $p_j\to \hat{p}$ as $j\to \infty$ via a subsequence for some $\hat{p}\geq 0$. The relation \eqref{equiv-sm eq3} yields $\hat{p}>p^*$. In particular, we have $H(x_0, t_0, \hat{p})>H(x_0, t_0, 0)$.

    Recalling the property \eqref{legendre}, we can find a sequence $q_j\geq 0$ such that
    \begin{equation}\label{equiv-sm eq4}
    H(x_0,t_0,p_j)\leq p_j q_j-L(x_0,t_0,q_j)+\dfrac{1}{j}. 
    \end{equation}
    Note that there exists $\hat{q}>0$ such that $q_j\geq \hat{q}$ for all $j\geq 1$ large, for otherwise we have a subsequence of $q_j$ converging to $0$, which yields a contradiction 
    \[
    H(x_0, t_0, \hat{p})\leq -L(x_0, t_0, 0)=H(x_0, t_0, 0). 
    \]
    Next, we choose $t_j<t_0$ such that $d(x_0,x_j)=q_j(t_0-t_j)$. It is clear that $t_j\to t_0-$ as $j\to\infty$. Moreover, noting that
    \[
    \begin{aligned}
    \dfrac{\psi(x_0,t_0)-\psi(x_j,t_j)}{t_0-t_j} &=\dfrac{\psi(x_j,t_0)-\psi(x_j, t_j)}{t_0-t_j}+\dfrac{\psi(x_0,t_0)-\psi(x_j, t_0)}{t_0-t_j} \\
    &\geq \dfrac{\psi(x_j,t_0)-\psi(x_j, t_j)}{t_0-t_j}+p_jq_j,
    \end{aligned}
    \]
    by \eqref{equiv-sm eq4} and \eqref{equiv-sm eq2} we have 
    \[
    \begin{aligned}
       \dfrac{\psi(x_j,t_0)-\psi(x_j, t_j)}{t_0-t_j}+H(x_0,t_0,p_j)& \leq \dfrac{\psi(x_j,t_0)-\psi(x_j, t_j)}{t_0-t_j}+p_jq_j-L(x_0,t_0,q_j)+\dfrac{1}{j}\\
      & \leq \dfrac{v(x_0,t_0)-v(x_j, t_j)}{t_0- t_j}-L\left(x_0,t_0,\frac{d(x_0,x_j)}{t_0-t_j}\right)-k+\dfrac{1}{j}.
    \end{aligned}
    \]
    Letting $j\to \infty$, we are led to \eqref{equiv-sm eq5} again. This concludes the proof.
\end{proof}

The reverse implication for Proposition \ref{semi-eqiuv_1} is not clear to us. It seems challenging to directly show that any slope-based subsolution (resp., slope-based supersolution) is a Monge subsolution (resp., Monge supersolution). 
However, thanks to the comparison principle in Theorem~\ref{unique_slope} for slope-based solutions, we are able to show the equivalence of the two notions of solutions for the initial value problem. We restrict our study of \eqref{eq time-ind} with $H$ in the form \eqref{special-H} satisfying all of the assumptions in Theorem~\ref{unique_slope}.

\begin{thm}[Equivalence of Monge and slope-based solutions for superlinear Hamiltonians]\label{thm:equivalence1}
    Let $(\X,d)$ be a complete length space and $T>0$. Let $H$ be given by \eqref{special-H} with $\tilde{H}$ satisfying the conditions (1)--(4) in Theorem~\ref{unique_slope}. Assume in addition that $u_0: \X\to \R$ is bounded Lipschitz and $f: \X\to \R$ is bounded uniformly continuous. Suppose that $u\in C(\X\times [0, T))$ satisfies \eqref{initial} in the sense of \eqref{initial-conti0}. Then, $u$ is a Monge solution of \eqref{eq time-ind} if and only if it is a slope-based viscosity solution of \eqref{eq time-ind}. 
\end{thm}

\begin{proof} Under the assumptions of Theorem~\ref{unique_slope}, we show that $H$ and the associated Lagrangian $L$ satisfy (H1)--(H4). Indeed, (H1) follows immediately from conditions (1)--(2) of Theorem~\ref{unique_slope} and the boundedness of $f$ in $\X$. Moreover, it is easily seen that
\begin{equation}\label{special-L-est0}
L(x, q)=\sup_{p\geq 0} \left\{pq -\tilde{H}(x, p)\right\}+f(x),
\end{equation}
which in view of \eqref{special-H-est1} yields
\begin{equation}\label{special-L-est1}
\tilde{C_1} q^{\beta}+f(x)\leq L(x, q)\leq \tilde{C_2}q^{\beta}+f(x)
\end{equation}
for all $(x, q)\in \X\times [0, \infty)$, where $\beta:=\alpha/(\alpha-1)>1$ and $\tilde{C}_1, \tilde{C_2}>0$ are constants independent of $x, p$. It is then clear that $L$ satisfies (H2)  with 
\begin{equation}\label{special-L-est4}
m(q)=\tilde{C_1} q^{\beta}+\inf_{\X} f, \quad q\geq 0.     
\end{equation}
Moreover, (H3) holds due to \eqref{special-L-est1} and the boundedness of $f$. The local uniform continuity of $L$ is a consequence of \eqref{special-L-est1} and the local uniform continuity of $H$. To see this, take arbitrarily $(x_1, q_1), (x_2, q_2)\in \X\times [0, \infty)$ with $d(x_1, x_0)\leq R, d(x_2, x_0)\leq R$,  $q_1, q_2\in [0, R]$ for fixed $R>0$ and $x_0\in \X$. 

By \eqref{special-L-est0}, for any $\vep>0$, there exists $p_1\geq 0$ such that 
\[
L(x_1, q_1)\leq p_1 q_1-\tilde{H}(x_1, p_1)+f(x_1)+\vep.
\]
Since $L(x_1, q_1)\geq \inf_X f$, we deduce that 
\[
\inf_X f\leq p_1 q_1-\tilde{H}(x, p_1)+\sup_\X f +\vep,
\]
which by \eqref{special-H-est1} again yields $p_1\leq C_R:=C(R^{\frac{1}{\alpha-1}}+1)$ for some constant $C>0$ depending on $\alpha>1$ and the uniform bound of $|f|$ in $\X$. Then, the local uniform continuity of $H$ implies the existence of a modulus of continuity $\omega_{H, R}$ such that 
\[
|H(x_1, p_1)-H(x_2, p_1)|\leq \omega_{H, R}(d(x_1, x_2))
\]
It follows from the definition of $L$ again that 
\begin{equation}\label{special-L-est2}
L(x_1, q_1)- L(x_2, q_2) \leq p_1 q_1-p_1q_2-\tilde{H}(x_1, p_1)+\tilde{H}(x_2, p_1)+f(x_1)-f(x_2)+\vep, 
\end{equation}
which yields
\[
\begin{aligned}
L(x_1, q_1)- L(x_2, q_2) & \leq C_R|q_1-q_2|+|{H}(x_1, p_1)-{H}(x_2, p_1)|+\vep\\
& \leq C_R|q_1-q_2|+\omega_{H, R}(d(x_1, x_2))+\vep. 
\end{aligned}
\]
Letting $\vep\to 0$ and exchanging the roles of $(x_1, q_1)$ and $(x_2, q_2)$, we obtain the local uniform continuity of $L$. The condition \eqref{L-lip1} can be derived from \eqref{special-H-est2} in the condition (4) and the uniform continuity of $f$ in $\X$. Indeed, for any $(x_1, q), (x_2, q)\in \X\times [0, \infty)$ and any $\vep>0$, following the arguments above for \eqref{special-L-est2} we can find 
\begin{equation}\label{special-L-est3}
p\leq C(q^{\frac{1}{\alpha-1}}+1)
\end{equation}
with some $C>0$ such that 
\[
L(x_1, q)- L(x_2, q) \leq -\tilde{H}(x_1, p)+\tilde{H}(x_2, p)+f(x_1)-f(x_2)+\vep.
\]
By \eqref{special-H-est2} and the uniform continuity of $f$, we thus have 
\[
L(x_1, q)- L(x_2, q)\leq \tilde{\omega}(d(x_1, x_2))(1+p^\alpha)+\omega_f(d(x_1, x_2))+\vep,
\]
where $\omega_f$ denotes the modulus of continuity of $f$. Applying \eqref{special-L-est3} and letting $\vep\to 0$, we are led to 
\[
L(x_1, q)- L(x_2, q)\leq \tilde{\omega}(d(x_1, x_2))\left(1+C^\alpha\left(q^{\frac{1}{\alpha-1}}+1\right)^\alpha\right)+\omega_f(d(x_1, x_2)).
\]
Recall that $\beta=\alpha/(\alpha-1)$. In view of \eqref{special-L-est4} and the arbitrariness of $x_1, x_2\in \X$, $q\geq 0$, we complete the proof of \eqref{L-lip1} by taking $\omega_L=C' \tilde{\omega}+ \omega_f$ with $C'>0$ sufficiently large. Hence, $H$ and $L$ (both independent of $t$) satisfy the assumptions (H1)--(H4). 

By Proposition \ref{semi-eqiuv_1}, any Monge solution of \eqref{eq time-ind} is a slope-based solution of \eqref{eq time-ind}. It follows from Theorem~\ref{unique_slope} that with initial condition interpreted as \eqref{initial-conti0}, any Monge solution of \eqref{eq time-ind}\eqref{initial} must be the unique slope-based viscosity solution of the initial value problem. 
\end{proof}

\begin{rmk}
    We have shown above that the Hamiltonian $H$ discussed in Theorem~\ref{unique_slope} satisfies our assumptions (H1)--(H4). As a weaker set of assumptions, (H1)--(H4) allows more general Hamiltonians even in the $t$-indepedent case. One example of $H$ that satisfies (H1)--(H4) but is not covered by Theorem~\ref{unique_slope} is 
    \[
    H(x, p)=a(x) p^\alpha+b(x) p-f(x), \quad x\in \X, p\geq 0, 
    \]
    where $a, b, f: \X\to \R$ are bounded and uniformly continuous such that $\inf_\X a>0$ and $b\geq 0$ in $\X$. This highlights that our notion and well-posedness results for Monge solutions apply to a broader class of superlinear Hamilton-Jacobi equations than those considered in \cite{GaS}. 
    \end{rmk}

As an immediate consequence of Theorem~\ref{thm:equiv-cm} and Theorem~\ref{thm:equivalence1}, we obtain the equivalence between curve-based solutions, slope-based solutions, and Monge solutions of \eqref{eq time-ind} satisfying \eqref{initial-conti0} under the assumptions of Theorem~\ref{unique_slope}. In particular, the equivalence result in Theorem~\ref{thm:equiv-power} holds. One can easily verify that the power-type Hamiltonian for \eqref{eq power} introduced in Theorem~\ref{thm:equiv-power} satisfies the assumptions in Theorem~\ref{unique_slope}. 

We remark that the property \eqref{H-cocercive} is not used in the proof of Proposition~\ref{semi-eqiuv_1}. Consequently, we can similarly show that any Monge solution of \eqref{eq special} is a slope-based solution. The definition of Monge solutions to \eqref{eq special} is given in Definition~\ref{def monge special}. Adopting the comparison principle in Theorem~\ref{thm:comparison-eikonal}, we obtain the equivalence between Monge solutions and slope-based solutions of the initial value problem for \eqref{eq special}. 

\begin{thm}[Equivalence of Monge and slope-based solutions of eikonal-type equation]\label{thm:equiv-eikonal2}
    Let $(\X,d)$ be a complete length space and $T>0$. Assume that $f: \X\to \R$ is bounded and continuous. Let $u_0\in\lip(\X)$ be bounded and $u\in C(\X\times [0, T))$ satisfy \eqref{initial} in the sense of \eqref{initial-conti0}. Then, $u$ is a Monge solution of \eqref{eq special0} if and only if it is a slope-based viscosity solution of \eqref{eq special0}. 
\end{thm}

Under the same assumptions, we conclude the equivalence of curve-based, slope-based, and Monge solutions to the initial value problem for \eqref{eq special0} by combining the results in Theorem~\ref{thm:equiv-eikonal1} and Theorem~\ref{thm:equiv-eikonal2}. 
It is of our interest for future work to consider the local equivalence between these notions of solutions to the equation without relying on the initial value.

\bibliographystyle{abbrv}

\end{document}